\documentclass[11pt,a4paper]{article}

\usepackage{amsmath}
\usepackage{amsfonts}
\usepackage{amsthm}
\usepackage{amssymb}
\usepackage{mhequ}
\usepackage{enumitem}
\usepackage{graphicx}
\usepackage{mathrsfs}
 \usepackage[hyperref]{xcolor}
 \usepackage[hidelinks]{hyperref}
\usepackage{fullpage}

\usepackage{pgfplots}
\pgfplotsset{compat=1.8}
\usetikzlibrary{math,calc,arrows.meta}
\usepgfplotslibrary{groupplots}

\usepackage{xpatch}
\tracingpatches
\xpatchcmd{\proof}{\itshape}{\bf}{}{}
\xpatchcmd{\example}{\itshape}{\bf}{}{}

\numberwithin{equation}{section}

\newtheorem{thm}{Theorem}[section]
\newtheorem{lemma}[thm]{Lemma}
\newtheorem{prop}[thm]{Proposition}
\newtheorem{cor}[thm]{Corollary}

\theoremstyle{definition}
\newtheorem{rmk}[thm]{Remark}
\newtheorem{defn}[thm]{Definition}

\theoremstyle{remark}
\newtheorem{example}[thm]{Example}

\newcommand{\R}{{\mathbb R}}
\newcommand{\Z}{{\mathbb Z}}

\newcommand{\bbS}{{\mathbb S}}

\DeclareMathOperator{\E}{\mathbb{E}}
\let\P\undefined
\DeclareMathOperator{\P}{\mathbb{P}}

\newcommand{\bY}{{\bar Y}}

\newcommand{\bF}{{\bar F}}
\newcommand{\bA}{{\bar A}}

\newcommand{\cA}{{\mathcal A}}
\newcommand{\cB}{{\mathcal B}}
\newcommand{\cD}{{\mathcal D}}
\newcommand{\cF}{{\mathcal F}}

\newcommand{\cN}{{\mathcal N}}
\newcommand{\cC}{{\mathcal C}}

\newcommand{\cQ}{{\mathcal Q}}
\newcommand{\cS}{{\mathcal S}}
\newcommand{\cW}{{\mathcal W}}
\newcommand{\cJ}{{\mathcal J}}
\newcommand{\cM}{{\mathcal M}}
\newcommand{\cR}{{\mathcal R}}

\newcommand{\hE}{{\widehat E}}
\newcommand{\hL}{{\hat L}}
\newcommand{\hW}{{\widehat W}}
\newcommand{\hX}{{\widehat X}}

\newcommand{\tH}{{\widetilde H}}
\newcommand{\tL}{{\tilde L}}
\newcommand{\tQ}{{\widetilde Q}}
\newcommand{\tR}{{\widetilde R}}

\newcommand{\tW}{{\widetilde W}}
\newcommand{\tZ}{{\widetilde Z}}

\newcommand{\var}{\text{-}\mathrm{var}}
\newcommand{\id}{\mathrm{Id}}

\newcommand{\eps}{\varepsilon}
\newcommand{\MM}{e}
\newcommand{\ELL}{{k_*}}
\newcommand{\Cov}{\operatorname{Cov}}
\newcommand{\Leb}{\operatorname{Leb}}
\newcommand{\Lip}{\operatorname{Lip}}
\newcommand{\sgn}{\operatorname{sgn}}

\newcommand{\SMALL}{\textstyle}

\newcommand{\dsim}{\sim_F}

\newcommand{\bone}{\mathbf{1}}

 \usepackage{stmaryrd}
 \newcommand{\disc}[1]{{\talloblong #1 \talloblong}}

\usepackage{mathrsfs}
\newcommand{\DD}{{\mathscr{D}}}

\title{Superdiffusive limits beyond the Marcus regime for deterministic fast-slow systems}

\author{Ilya Chevyrev\thanks{School of Mathematics, The University of Edinburgh, Edinburgh, EH9 3FD, UK.
\href{mailto:ichevyrev@gmail.com}{\tt ichevyrev@gmail.com}}
\and Alexey Korepanov\thanks{Department of Mathematical Sciences, Loughborough University,
Loughborough, Leicestershire, LE11 3TU, UK. \href{mailto:a.korepanov@lboro.ac.uk}{\tt khumarahn@gmail.com}}
\and Ian Melbourne\thanks{Mathematics Institute,
University of Warwick,
Coventry, CV4 7AL,
UK.
\href{mailto:i.melbourne@warwick.ac.uk}{\tt i.melbourne@warwick.ac.uk}}}

\date{4 January 2024; Revised 21 October 2024}

\begin{document}

\maketitle

\begin{abstract}
We consider deterministic fast-slow dynamical systems of the form
\[
x_{k+1}^{(n)} = x_k^{(n)} + n^{-1} A(x_k^{(n)}) + n^{-1/\alpha} B(x_k^{(n)}) v(y_k), \quad y_{k+1} = Ty_k,
\]
where $\alpha\in(1,2)$ and $x_k^{(n)}\in\R^m$. Here, $T$ is a slowly mixing nonuniformly hyperbolic dynamical system and the process $W_n(t)=n^{-1/\alpha}\sum_{k=1}^{[nt]}v(y_k)$ converges weakly to a $d$-dimensional $\alpha$-stable  L\'evy process $L_\alpha$.

We are interested in convergence of the $m$-dimensional process $X_n(t)=  x_{[nt]}^{(n)}$ to the solution of a stochastic differential equation (SDE)
\[
 dX = A(X)\,dt + B(X)\, dL_\alpha.
\]
In the simplest cases considered in previous work, the limiting SDE has the Marcus interpretation. In particular, the SDE is Marcus if the noise coefficient $B$ is exact or if the excursions for $W_n$ converge to straight lines as $n\to\infty$.

Outside these simplest situations, it turns out that  typically the Marcus interpretation fails. We develop a general theory that does not rely on exactness or linearity of excursions. To achieve this, it is necessary to consider suitable spaces of ``decorated'' c\`adl\`ag paths and to interpret the limiting decorated SDE. In this way, we are able to cover more complicated examples such as billiards with flat cusps where the limiting SDE is typically non-Marcus for $m\ge2$.
\end{abstract}

\renewcommand{\thefootnote}{\fnsymbol{footnote}} 

\renewcommand{\thefootnote}{\arabic{footnote}}

\tableofcontents

\section{Introduction}
\label{sec:intro}

Homogenisation for systems with multiple timescales is
a longstanding and very active area of research~\cite{PavliotisStuart}.  
Recently there has been considerable interest in the case where the underlying multiscale system is deterministic, see~\cite{ CFKM20,CFKMZ22,Dolgopyat04,Dolgopyat05,
GM13,KM16,KM17,KKM22,MS11,Pene02} as well as our survey paper~\cite{CFKMZ}.
Although the system is deterministic, chaotic dynamics often leads to a limiting stochastic differential equation (SDE) driven by Brownian motion.  

We are interested in the case where the limiting SDE is driven by a superdiffusive $\alpha$-stable L\'evy process. This situation was considered previously in~\cite{CFKM20,GM13} though, as we shall see, under restrictions that gave a misleadingly simple picture.

\subsection{Background material on interpretation of stochastic integrals}
When analysing SDEs, a key question is the appropriate interpretation of certain stochastic integrals. Consider the SDE
\begin{equation} \label{eq:SDE}
dX=A(X)\,dt+B(X)\,dW, \quad X(0)=\xi,
\end{equation}
where $A:\R^m\to\R^m$, $B:\R^m\to\R^{m\times d}$ are smooth, $\xi\in\R^m$, and $W$ is a continuous $d$-dimensional stochastic process. If $W$ is smooth, then~\eqref{eq:SDE} reduces to an ordinary differential equation (ODE). If $W$ is $\eta$-H\"older with $\eta>\frac12$, or alternatively of finite $p$-variation with $p<2$, then the integral $\int B(X)\,dW$ is well-defined as a Riemann-Stieltjes integral and is called a Young integral~\cite{Young36}. Unfortunately, Brownian motion is $\eta$-H\"older only for $\eta<\frac12$ and is of finite $p$-variation only for $p>2$. Hence, when $W$ is Brownian motion, $\int B(X)\,dW$ is not generally well-defined as a Young integral, and the question arises as to whether the correct interpretation is It\^{o} (which has the advantage of preserving the property of being a martingale), Stratonovich (which has the advantage of transforming under the standard laws of calculus) 
or some other interpretation.
Wong and Zakai~\cite{WZ65} attempted to resolve the issue of the interpretation of stochastic integrals by considering smooth approximations whereby smooth processes $W_n$ converge weakly to $W$ as $n\to\infty$. The SDE~\eqref{eq:SDE}
is then replaced by the ODE
\begin{equation} \label{eq:ODE}
dX_n=A(X_n)\,dt+B(X_n)\,dW_n, \quad X_n(0)=\xi,
\end{equation}
 and it is hoped that computing the weak limit $X$ of $X_n$ will help to clarify the nature of the limiting stochastic integral. Under certain conditions, it was shown in~\cite{WZ65}
that $X_n\to_w X$ where $X$ satisfies~\eqref{eq:SDE} with the Stratonovich interpretation. The conditions in~\cite{WZ65} hold for $d=m=1$ and more generally under certain exactness conditions (for example, it suffices that $d=m$ and $B^{-1}=dr$ for some diffeomorphism $r:\R^m\to\R^m$), but counterexamples for $d=m=2$ were subsequently found by~\cite{McShane72,Sussmann91}.

When exactness fails, a fundamental issue is the lack of continuous dependence of the outputs $X_n$ on the inputs $W_n$ in equation~\eqref{eq:ODE}. The theory of rough paths developed by Lyons~\cite{Lyons98} demonstrates that the answer is to view $X_n$ as a function of processes $(W_n,\int W_n \otimes\, dW_n)$ taking values in $\R^d\times\R^{d\times d}$. Working in the uniform topology on continuous paths in $\R^d\times\R^{d\times d}$
is still insufficient, and it is necessary to consider $p$-variation~\cite{Lyons98} or H\"older topologies~\cite{Friz2005,FrizHairer}.
This leads to nontrivial corrections to the Stratonovich integral via the ``L\'evy area'' given by the weak limit of $\int W_n \otimes\,dW_n$.
In~\cite{KM16}, see also~\cite{KM17,CFKMZ,CFKMZ22},
it was shown using rough path theory that a large class of deterministic systems (maps or ODEs) give rise to smooth approximations of Brownian motion with well-defined and nontrivial L\'evy area; the correct interpretation is generally neither It\^{o} nor Stratonovich. 

\subsection{Stochastic integration with respect to L\'evy processes}

Now we turn to the case where continuous paths $W_n$ converge weakly to $W$ which is
an $\alpha$-stable L\'evy process, $\alpha\in(0,2)$.
We write $L_\alpha$ instead of $W$ in this case, so the limiting SDE becomes
\begin{equation} \label{eq:SDEL}
    dX = A(X)\,d t + B(X) \,d L_\alpha, \quad X(0) = \xi.
\end{equation}
The L\'evy process $L_\alpha$ is tight in $p$-variation for some $p < 2$.
Were $L_\alpha$ continuous, the SDE~\eqref{eq:SDEL} would have a comparatively simple
interpretation in the sense of Young as shown in \cite{Lyons94}, but the discontinuities of $L_\alpha$
leave~\eqref{eq:SDEL} open to interpretations.

Again, the simplest situation is when $B$ is exact; there it was shown in~\cite[Section~5]{GM13} that the limiting SDE has the Marcus interpretation~\cite{Marcus80} (see~\cite{KPP95,Applebaum09,CP14} for the general theory of Marcus SDEs and their applications).
We note that the Marcus integral 
plays the same role for L\'evy processes that Stratonovich does for Brownian motion:
it is the integral that transforms under the standard laws of calculus.

It turns out that the Marcus integral may give the correct interpretation beyond the exact situation although the mechanism is somewhat different. 
In~\cite{CFKM20}, we gave general conditions under which there is convergence to a Marcus SDE, but a crucial assumption was that $W_n$ converges to $L_\alpha$ in the $\cM_1$ Skorokhod topology
on the Skorokhod space $D([0,1], \R^m)$ of c\`adl\`ag functions.

To  circumvent nonexactness, we used topologies inspired by rough paths:
weak convergence of $W_n$ no longer suffices and
we required in addition that $W_n$ is tight in $p$-variation for some $p<2$.
Moreover, although convergence of $W_n$ was assumed in the $\cM_1$ topology, $X_n$ need not converge in any of the
Skorokhod topologies~\cite{S56} even
in the simplest situations~\cite{CFKM20}. Hence we 
 needed to consider convergence of $X_n$ in generalised Skorokhod topologies as introduced recently in Chevyrev and Friz~\cite{CF19}.

In Theorem~\ref{thm:Marcus}, we extend the result in~\cite{CFKM20} to include general ``linear excursions'' for $W_n$ (by which we mean that each excursion is asymptotically a line segment, possibly with overshoots).
This includes the situations where $W_n$ converges weakly in the $\cM_1$ or $\cM_2$ topologies.

More importantly, in our main result Theorem \ref{thm:X}, we are able to show convergence of $X_n$ when $W_n$ does not have linear excursions and thus fails to converge in the $\cM_1$ or $\cM_2$ topologies.
When $d\ge2$, linear excursions for $W_n$ are highly unlikely 
unless there is a very simple mechanism underlying the excursions, such as a neutral fixed point for an intermittent map.
Moreover, as shown in Example~\ref{ex:non-Marcus} below,
typically the limiting SDE is not Marcus when $B$ is not exact and the excursions are not linear.

\begin{example} \label{ex:non-Marcus}
We show that in the simplest nonexact examples, convergence to a Marcus integral typically fails unless the excursions of $W_n$ are linear.
Take $d=m=2$, $A\equiv0$, $\xi=0$ and
\[
B:\R^2\to \R^{2\times 2},
\qquad B(x_1,x_2)=\left(\begin{array}{cc} 1 & 0 \\ 0 & x_1
\end{array}\right).
\]
Then~\eqref{eq:SDEL} becomes
\[
        dX^1 = dL^1, \qquad
        dX^2 = X^1 dL^2, \qquad X(0) = 0.
\]
In particular, $X^1(t)=L^1(t)$ and $X^2(t)=\int_0^t L^1(t)\,dL^2(t)$.

Choose $h=(h_1,h_2) \in C^1([0,1], \R^2)$ where $h(0)=(0,0)$, $h(1)=(1,1)$. 
We define $W_n \in D([0,1], \R^2)$ by
\[
    W_n(t)
    = \begin{cases}
        (0,0) & \text{if } 0  \le t < \frac12 - \frac1n  \\
        h\big(n(t-\frac12+\frac1n)\big) & \text{if } \frac12 -\frac1n \le t <\frac12 \\
        (1,1) & \text{if } \frac12  \le t \le 1
    \end{cases} .
\]
Then $W_n$ converges pointwise to $L\in D([0,1],\R^2)$ given by
$L(t) = \bone_{t \ge \frac{1}{2}}(1,1)$.
Moreover, a direct calculation shows that
the solutions $X_n$ to the ODE $dX_n=B(X_n)\,dW_n$
converge pointwise to $X\in D([0,1],\R^2)$ where
$X(t)=\bone_{t\ge\frac12}\big(1,\,\int_0^1 h_1(s)\,h_2'(s)\,ds\big)$.

We distinguish the case where the excursion
of $W_n$ is linear, so  $\{h(s);\,0\le s\le 1\}$ is contained in a line. In other words, $h_1\equiv h_2$. This includes $\cM_1$ and $\cM_2$ convergence as special cases. Indeed
$W_n$ converges to $L$ 
in the $\cM_2$ topology if and only if $\{h(s);\,0\le s\le 1\}$ is the line segment joining $(0,0)$ to $(1,1)$.
For convergence in the $\cM_1$ topology, it is required moreover that $h_1$ is non-decreasing.

In the linear case $h_1\equiv h_2$, we obtain by a change of variables that
$X(t)=\bone_{t\ge\frac12}\big(1,1)$
and the stochastic integral
$dX = B(X) \, dL$ has the Marcus interpretation.
Otherwise, even though the endpoints of the excursion for $W_n$ on $[\frac{1}{2}-\frac1n,\frac{1}{2}]$
are fixed and the limiting path $L$ is fixed, the values of the limiting integral $X$ on $[\frac12,1]$ depend sensitively on the excursion $h$.
\end{example}

\vspace{-1ex}
\paragraph{Decorated paths}
Example~\ref{ex:non-Marcus} shows that it is necessary to include excursions in the description of the limiting path $L$. Hence it is necessary to consider spaces of decorated c\`adl\`ag paths where jumps are accompanied by ``decorations'' according to the shape of the excursion. Such spaces were introduced by Whitt~\cite{Whitt02} and studied further by Freitas \emph{et al.}~\cite{FFT,FFMT}. Here, we build on and generalise the spaces used by
Chevyrev~\emph{et al.}~\cite{Chevyrev18,CF19,CFKM20} which sufficed for the Marcus limits in~\cite{CFKM20}.
As in~\cite{CFKM20}, weak convergence of $W_n$ is insufficient to deduce convergence of $X_n$ and we also require tightness of $W_n$ in $p$-variation.

The decorated c\`adl\`ag space
$\DD$ that we introduce in Section~\ref{sec:DD} has several advantages.
For our purposes, the main advantage is that it is well-suited for studying fast-slow systems
because there is a natural notion of a solution to a
``decorated differential equation'' in $\DD$, which moves beyond linear excursions of the driver and retains continuity of the solution map $\DD\ni L_\alpha \mapsto X \in\DD$ to the SDE \eqref{eq:SDEL},
see Theorem \ref{thm:S}.
Our framework for interpreting \eqref{eq:SDEL} in the presence of jumps is essentially deterministic and has novelties even if $L_\alpha$ has bounded variation.
It notably unifies the Marcus (geometric/canonical) and It\^o (forward) interpretations of differential equations; see~\cite{CF19, FZ18} for details on these interpretations.

Another advantage is that $\DD$ distinguishes more decorated paths than the spaces from~\cite{Whitt02,FFT,FFMT} (see Example~\ref{ex:DDF'}).
In addition, our metric on $\DD$ is complete and separable whereas the metrics considered in~\cite{FFT,Whitt02} are not simultaneously complete and separable.
In this regard, it is worth recalling that the standard $\cJ_1$ metric $\sigma_\infty$~\eqref{eq:J1} on the c\`adl\`ag space $D$ is separable but not complete.
Despite the fact that our construction of $\DD$ as an
extension of the spaces in~\cite{Chevyrev18,CF19,CFKM20} is rather technical, we have the
following notable characterisation.

\begin{thm} \label{thm:comp}
$\DD$ is the completion of $D$ in the standard $\cJ_1$ metric $\sigma_\infty$.
\end{thm}

This result is proved in Section~\ref{subsec:SkorM1}.
In addition, $\DD$ simultaneously generalises the classical Skorokhod $\cJ_1$ and $\cM_1$ topologies via different embeddings of $D$ into $\DD$ (see Remark \ref{rem:M1}).

\subsection{Fast-slow systems}
Let $\alpha\in(1,2)$.
In this paper, we consider multiscale equations of the form
\begin{equation} \label{eq:fs}
  \begin{cases}
    x_{k+1}^{(n)} = x_k^{(n)} + n^{-1} A(x_k^{(n)}) + n^{-1/\alpha} B(x_k^{(n)}) v(y_k) , \\
    y_{k+1} = Ty_k
  \end{cases}
\end{equation}
defined on $\R^m\times \Lambda$ where $\Lambda$ is a bounded metric space.
Here,
\[
    A : \R^m\to\R^m
    , \quad
    B : \R^m\to\R^{m\times d}
    , \quad
    v : \Lambda\to\R^d
    , \quad
    T: \Lambda\to \Lambda
    .
\]
It is assumed that the fast dynamical system $T: \Lambda\to \Lambda$ has an ergodic invariant probability measure $\mu_\Lambda$.  
Also, $v : \Lambda\to\R^d$ is H\"older with $\int v\,d\mu_\Lambda=0$.
Define 
\begin{equation} \label{eq:Wn}
    W_n(t)=n^{-1/\alpha}\sum_{j=0}^{[nt]-1}v\circ T^j
    .
\end{equation}
Then $W_n$ belongs to $D([0,1], \R^d)$ and can be viewed
as a random process on the probability space $(\Lambda, \mu_\Lambda)$ depending on the initial condition $y_0\in \Lambda$.
We assume that $T$ and $v$ exhibit superdiffusive behaviour, so that
the 
process $W_n$ converges weakly in $D([0,1],\R^d)$ to an
$\alpha$-stable L\'evy process $L_\alpha$.

Now fix $x_0^{(n)}=\xi \in \R^m$, and solve~\eqref{eq:fs} to obtain
$(x_k^{(n)},y_k)$ depending on the initial condition $y_0\in(\Lambda,\mu_\Lambda)$.  Define the c\`adl\`ag process $X_n\in D([0,1],\R^m)$ given by
$X_n(t) = x_{[nt]}^{(n)}$;   again we view this as a process on $(\Lambda,\mu_\Lambda)$.
Our aim is to show, 
under mild regularity assumptions on the functions $A: \R^m\to \R^m$ and $B : \R^m \to \R^{m \times d}$, that
$X_n\to_{\mu_\Lambda} X$ \footnote{We write $\to_{\mu_\Lambda}$ to denote weak convergence, emphasising the probability space on which $X_n$ are defined. As in~\cite{CFKM20}, $\mu_\Lambda$ can be replaced by any probability measure absolutely continuous with respect to $\mu_\Lambda$.} where $X$ is the solution of the SDE~\eqref{eq:SDEL}.
A key part of the theory is to explain the sense in which the limiting SDE is to be interpreted.

\vspace{2ex}
In this paper, we present a general theory under which the slow dynamics converges to an SDE driven by $L_\alpha$ regardless of the nature of the convergence of $W_n$. 
Moreover, we give a full description of how to interpret the limiting SDE for both the Marcus and non-Marcus cases.

\subsection{Solving decorated SDEs}
A key part of our paper is to show how to solve equations driven by decorated processes. In the framework of Example~\ref{ex:non-Marcus}, we argue that $W_n$ converges to a decorated process $\hL$ consisting of $L$ together with an excursion $\{h(s);\,0\le s\le 1\}$ at the discontinuity point $t=\frac12$. 
Solving the equation $d\hX=B(\hX)\,d\hL$ yields a decorated process $\hX$, and we show that
$X_n$ converges weakly to $\hX$. 

To solve
$d\hX=B(\hX)\,d\hL$ we proceed as follows: first, we insert an interval $[\frac12,\frac12+\delta]$ into the interval $[0,1]$ and define a \emph{$\delta$-extension} $L^\delta\in D([0,1+\delta],\R^2)$ coinciding with $L$ on $[0,\frac12]\cup[\frac12+\delta,1+\delta]$ (suitably displaced) and given by $h$ (suitably scaled) on $[\frac12,\frac12+\delta]$.
Then we solve the Young equation $X^\delta = B(X^\delta)\,dL^\delta$ for 
$X^\delta\in D([0,1+\delta],\R^2)$.
Finally, we collapse the inserted interval $[\frac12,\frac12+\delta]$ back to a point $\frac12$ to obtain the desired decorated limit $\hX$ consisting of a c\`adl\`ag path on $[0,1]$ and an excursion at $t=\frac12$.
Precise definitions are given in Section~\ref{sec:paths}.

\vspace{1ex}
\begin{example}
\label{ex:cusps}

A motivating example for the results in this paper is a class of dispersing billiards with flat cusps introduced by~\cite{JungZhang18}.
The billiard table $Q\subset\R^2$ has a boundary consisting of $C^3$ curves
with at least one flat cusp. In local
coordinates $(s, z) \in\R^2$, the $i$'th cusp lies at $(0, 0)$ and the bounding curves are given by
$\Gamma_\pm = \{(s, \pm s^{\beta_i} )\}$
close to $(0, 0)$,
where $\beta_i > 2$. We let $\beta=\max\beta_i$ corresponding to the flattest cusp(s).

The phase space of the billiard map (collision map) $T$ is given by $\Lambda= \partial  Q \times[-\pi/2, \pi/2 ]$,
with coordinates $(r, \theta )$ where $r$ denotes arc length along $\partial  Q$ and $\theta$ is the angle between
the normal to the boundary and the collision vector in the clockwise direction. (A standard reference for billiard maps is~\cite{ChernovMarkarian}.)
There
is a natural ergodic $T$-invariant probability measure $d\mu_\Lambda = (2|\partial Q|)^{-1} \cos \theta \,dr \,d\theta$ on $\Lambda$,
where $|\partial Q|$ is the length of $\partial Q$.

Let $v:\Lambda\to\R$ be a H\"older observable with $\int_\Lambda v\,d\mu_\Lambda=0$ and define $W_n$ as in~\eqref{eq:Wn}.
By~\cite{JungZhang18,JungPeneZhang20}, $W_n(1)$ converges weakly to an $\alpha$-stable law with $\alpha=\beta/(\beta-1)\in(1,2)$.
Convergence to the corresponding L\'evy process is considered in~\cite{MV20,JungPeneZhang20,JMPVZ21}.
Such convergence is impossible in the standard Skorokhod $\cJ_1$ topology since the jumps are bounded.
If $v$ has constant sign on each cusp, then convergence holds in the $\cM_1$ topology.
For more general $v$, convergence of $W_n$ might hold only in the $\cM_2$ topology~\cite{MV20} or may fail in any Skorokhod topology~\cite{JMPVZ21,MV20}. Indeed, the latter is typically the case when $d\ge2$.  Nevertheless, convergence always holds in a suitable decorated path space~\cite{FFMT}.

For purposes of illustration, consider the case of a billiard with a single flat cusp as shown in Figure \ref{fig:billiard}.
\begin{figure}[h]
\centering
        \begin{tikzpicture}[
            scale=4,
            myarrow/.style={-{Triangle[length=3mm,width=1.5mm]}},
            declare function = {
                fpx(\x) = - \x^2 / sqrt(1 + \x^4);
                fpy(\x) =     1  / sqrt(1 + \x^4);
            }
            ]
            \fill[gray!20]
                plot[domain=0:1, samples=40] ({\x + 0.05 * fpx(\x)}, {  \x^3 / 3 + 0.05 * fpy(\x) } )
                -- (1.05,1/3) -- (1.05,-1/3) -- 
                plot[domain=0:1, samples=40] ({(1-\x) + 0.05 * fpx(1-\x)}, {- (1-\x)^3 / 3 - 0.05 * fpy(1-\x) } )
                -- cycle;
            \fill[gray!20] (0,0) circle (0.05) (1,1/3) circle (0.05) (1,-1/3) circle (0.05);
            \filldraw[color=gray,thick,fill=white]
                plot[domain=0:1, samples=40] ({\x}, {  \x^3 / 3} )
                -- (1,-1/3)
                plot[domain=0:1, samples=40] ({(1-\x)}, {- (1-\x)^3 / 3} );
            \draw[myarrow](0.81,-0.177147)--(0.85,0.05);
            \draw[fill,black] (0.81,-0.177147) circle [radius=0.02];
        \end{tikzpicture}
        \caption{Billiard with a single flat cusp.}
        \label{fig:billiard}
\end{figure}
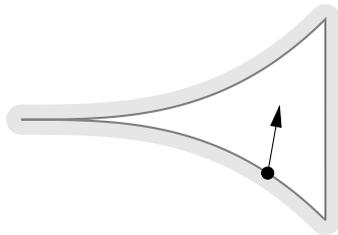

Take $\beta=3$, $\alpha=\frac32$ and
$v:\Lambda\to\R^2$ to be the mean zero observable
$v(r, \theta) = (\cos 3 \theta, \cos 5 \theta )$. 
In Figure \ref{fig:excursions}, we show
a parametric plot of a computer-simulated $W_n(t)$, $0 \le t \le 1$,
with separate plots for its coordinates $W_n^1(t)$, $W_n^2(t)$.
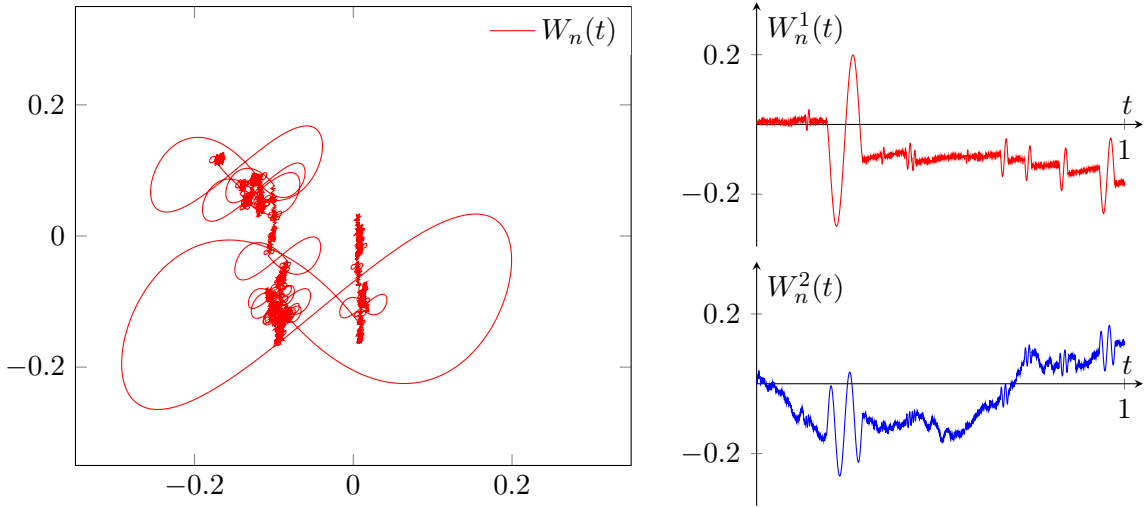
\begin{figure}[h]
\centering
    \begin{tikzpicture}
        \begin{axis}[
            name=W,
            width = 3.5in,
            xticklabel style={/pgf/number format/fixed},
            yticklabel style={/pgf/number format/fixed},
            legend style={at={(1,1)},anchor=north east},
            legend columns=-1,
            legend style={draw=none},
            xmin=-0.35, xmax=0.35,
            xtick distance=0.2,
            ymin=-0.35, ymax=0.35,
            ytick distance=0.2
            ]
            \addplot[red] table [x=x, y=y] {butterfly.data};
            \legend{$W_n(t)$}
        \end{axis} 
    \end{tikzpicture}
    \quad
    \begin{tikzpicture}
		\begin{groupplot}[
			group style={group size=1 by 2, vertical sep=0.2cm, ylabels at=edge left},
			width=2.0in,
			height = 1.275in,
			scale only axis,
			ymin=-0.35,ymax=0.35,
			ytick distance=0.2,
	        xmax = 1.05,
	        xlabel = {$t$},
	        xtick = {0, 1.0},
	        axis lines = center,
			]
			\nextgroupplot[ylabel = {$W_n^1(t)$}]
		        \addplot[red]  table [x=t, y=x] {butterfly.data};
			\nextgroupplot[ylabel = {$W_n^2(t)$}]
		        \addplot[blue] table [x=t, y=y] {butterfly.data};
		\end{groupplot}
	\end{tikzpicture}
    \caption{Parametric plot of $W_n$ for the observable $v(r, \theta) = (\cos 3 \theta, \cos 5 \theta )$.
The data can be accessed at \url{https://khumarahn.github.io/CKM23}.}
    \label{fig:excursions}
\end{figure}

Every deep excursion into the cusp resembles the figure $\infty$ (or a butterfly)
with the start and end points close to the centre.
Despite the geometric simplicity of the example, such excursions are clearly incompatible with the Skorokhod topologies. Moreover, the excursions are not linear and hence outside the regime where one should expect a Marcus SDE.

Figure~\ref{fig:excursions} clearly shows the excursions (identical up to scaling) corresponding to a single profile associated with the cusp. There are also small-scale motions resulting from values of the observable $v$ away from the cusp. 
\end{example}

Our main results incorporate a large class of nonuniformly hyperbolic systems with superdiffusive behaviour, including billiards with flat cusps.
For such systems, we recover the result of~\cite{FFMT} that $W_n$ converges weakly to a decorated $\alpha$-stable L\'evy process. Moreover, we relax certain assumptions in~\cite{FFMT}: we do not require that endpoints of excursions are in distinct directions, nor that they are nonzero. In addition, we verify tightness of $W_n$ in $p$-variation (extending the corresponding result in~\cite{CFKM20} to a larger class of dynamical systems).
Theorem~\ref{thm:X} guarantees that these are the ingredients required to show weak convergence of $X_n$ to solutions of a correspondingly decorated SDE.

\vspace{1ex}

The remainder of this paper is organised as follows.
In Section~\ref{sec:paths}, we define our space of decorated c\`adl\`ag paths
and explain rigorously how to solve SDEs driven by such decorated paths.
Then in Section~\ref{sec:setup}, we state our main results.
Sections~\ref{sec:pre} to~\ref{sec:tight} are devoted to the proofs of these main results. In Section~\ref{sec:paths-pf}, we prove the results from Section~\ref{sec:paths} for the decorated path space.

\paragraph{Notation}
We use ``big O'' and $\ll$ notation interchangeably, writing $a_n=O(b_n)$ or $a_n\ll b_n$
if there are constants $C>0$, $n_0\ge1$ such that
$a_n\le Cb_n$ for all $n\ge n_0$.
As usual, $a_n=o(b_n)$ means that $a_n/b_n\to0$ and $a_n\sim b_n$ means that $a_n/b_n\to1$.

For a function $h: [a,b]\to \R^d$, we denote $h(t^-)=\lim_{s\uparrow t} h(s)$ whenever this limit exists.
We let $\to_w$ denote weak convergence of random variables.
We write $\Z^+ = \{1,2,\ldots\}$.

\section{Decorated path space and \texorpdfstring{$p$}{p}-variation}
\label{sec:paths}

In this section we define the space $\DD$ of decorated c\`adl\`ag paths with metric $\alpha_\infty$. We also define the enhanced space
$\DD^{p\var}$ with metric $\alpha_{p\var}$ based on $p$-variation.
We then show how to construct solutions $X$ to differential equations driven by elements $\phi\in\DD^{p\var}$,
and end with Theorem~\ref{thm:S} which states that the solution map $\phi\mapsto X=\cS(\phi)$ preserves weak convergence.

Proofs of the results in this section are delayed to Section~\ref{sec:paths-pf} where independence of definitions on various choices is also verified.

\subsection{Decorated path space}
\label{sec:DD}

Fix $d\ge1$.
Let $I=[a,b]\subset\R$ be a bounded closed interval.
We denote by $C(I,\R^d)$ and $D(I,\R^d)$ the spaces of continuous and c{\`a}dl{\`a}g functions 
equipped respectively with the uniform norm
$|h|_{\infty;I} = \sup_{x\in I}|h(x)|$
and the Skorokhod $\cJ_1$-metric
\begin{equation} \label{eq:J1}
\sigma_{\infty;I} (h_1,h_2) = 
\inf_{\rho \in \cR} \max\{|\rho-\id|_\infty , |h_1\circ \rho-h_2|_{\infty}\}.
\end{equation}
Here, $\cR$ is the set of continuous increasing bijections $\rho:I\to I$.
We will drop $I$ from our notation in norms and metrics whenever it is clear from the context.

In the remainder of this subsection, we suppress the dependence on $\R^d$, writing $C(I)$, $D(I)$ and so on.
Also, we write for instance $D[a,b]$ instead of $D([a,b])$

Consider a function $\phi:I\to D[0,1]$. 
If $\phi(t)$ is a constant path in $D[0,1]$, then we say that $t\in I$ is a \emph{stationary} point for $\phi$. 

\begin{defn} \label{def:DDbar}
We let $\bar\DD(I)$ denote the space of pairs
$(\phi,\Pi)$ where $\phi:I\to D[0,1]$ and
\begin{enumerate}[label=(\alph*)]
\item \label{pt:cadlag}
the function $t\mapsto \phi(t)(1)$ lies in $D(I)$,
\item $\Pi\subset I$ is an at most countable set (possibly empty) containing the non-stationary points of $\phi$,
\item\label{pt:non_osc} 
for all $\eps>0$, there exist only finitely many $t\in\Pi$ such that $\sup_{s \in [0,1]} |\phi(t)(s) -\phi(t)(0)| > \eps$.
\end{enumerate}
\end{defn}

Let $(\phi,\Pi)\in\bar\DD[a,b]$ 
and write $\Pi=\{t_j:1\le j\le \kappa\}$ where $\kappa\in\{0,1,\dots\}\cup\{\infty\}$.
Given $\delta>0$, we define
the \emph{$\delta$-extension} $\phi^\delta\in D[a,b+\delta]$ as follows.
When $\kappa=0$, we define $\phi^\delta(t)=\phi(t)$ for $t\in[a,b]$ and $\phi^\delta(t)=\phi(b)$ for $t\in[b,b+\delta]$.
Otherwise, define $r=\sum_{j=1}^\kappa2^{-j}>0$.
We insert in $[a,b]$ \emph{fictitious} intervals $I_j$ of length $2^{-j}\delta/r$ after each $t_j$ to obtain the interval $[a,b+\delta]$.
Define  $\phi^\delta:[a,b+\delta]\to\R^d$ to coincide with $\phi$
on $[a,b+\delta]\setminus\bigcup_j I_j$ and to coincide with the appropriate time-scaled version of $\phi(t_j)$ on $I_j$.
(So if $I_j=[c_j,c_j+2^{-j}\delta/r]\subset[a,b+\delta]$, then
$\phi^\delta(c_j+s)=\phi(t_j)(s2^j r/\delta )$ for $0\le s\le 2^{-j}\delta/r$.)
See Section~\ref{subsec:pathFuncs} for a more precise definition.
By Lemma~\ref{lem:cadlag_path}, $\phi^\delta\in D[a,b+\delta]$ for all $(\phi,\Pi)\in\bar\DD[a,b]$.

For
$(\phi_1,\Pi_1),\,
(\phi_2,\Pi_2)\in \bar\DD[a,b]$ we set
\[
\alpha_{\infty;[a,b]}(\phi_1,\phi_2)=\lim_{\delta\to0}\sigma_{\infty;[a,b+\delta]}(\phi_1^\delta,\phi_2^\delta).
\]
By Lemma~\ref{lem:limExists}\ref{pt:alpha_infinity}, this limit exists
and is independent of $\Pi_1$, $\Pi_2$ and the various other choices above.
From now on we speak of $\phi\in \bar\DD[a,b]$ and suppress the set $\Pi$.

We are now in a position to define our metric space of decorated c\`adl\`ag paths. 
Given $h_1,h_2:[0,1]\to \R^d$, we say that
$h_1$ is a \emph{reparametrisation} of $h_2$ if
$\inf_{\rho\in\cR}|h_1\circ\rho-h_2|_\infty=0$.
Say that $\phi_1,\,\phi_2\in\bar\DD[a,b]$
are \emph{equivalent} if
the paths $s\mapsto \phi_i(t^-)(1) \bone_{s=0} + \phi_i(t)(s) \bone_{s>0}$ for $i=1,2$ are reparametrisations of each other for all $t\in[a,b]$.
\begin{defn} \label{def:DD}
We let $\DD[a,b]=\DD([a,b],\R^d)$ be the space of equivalence classes in
$\bar\DD[a,b]$.
\end{defn}

By Lemma~\ref{lem:limExists}\ref{pt:alpha_infinity}
and Theorem~\ref{thm:alpha}, $\alpha_\infty$ is constant on equivalence classes and defines a complete separable metric on $\DD[a,b]$.
Elements of $\DD[a,b]$ are called \emph{decorated paths}, generalising a similar notion
introduced in \cite{Chevyrev18,CF19} called path functions.

\begin{example} \label{ex:DDF'}
Our space $\DD$ distinguishes more decorated paths than the space $F'$ considered in~\cite{FFT,FFMT}, which we recall is finer than the space $F$ in~\cite[Sec.~15.7]{Whitt02} (see~\cite[Section~1.1]{FFT}).
The following are examples of decorated paths that are distinct in $\DD$ but identified in the space $F'$ considered in~\cite{FFT,FFMT} (see Figure \ref{fig:distinct_paths} for illustrations):
\[
    \phi(t)(s) =
    \begin{cases}
        \sin(2 \pi s) & \text{if } t = 1/3 \\
        \sin(4 \pi s) & \text{if } t = 2/3 \\
        0 & \text{else}
    \end{cases}
    \quad \text{and} \quad
    \psi(t)(s) =
    \begin{cases}
        \sin(4 \pi s) & \text{if } t = 1/3 \\
        \sin(2 \pi s) & \text{if } t = 2/3 \\
        0 & \text{else}
    \end{cases}
\]
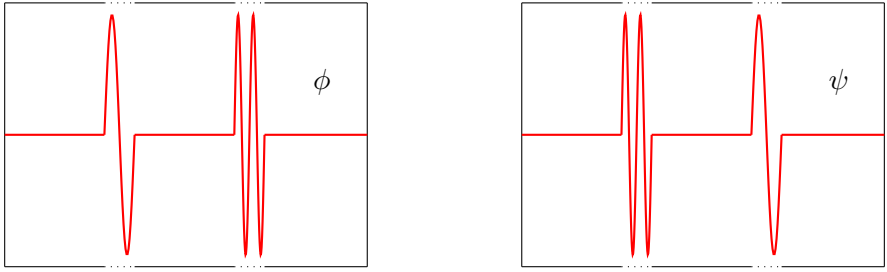
\begin{figure}[h]
\centering
    \begin{tikzpicture}[
        declare function = {
            xeeps(\x,\pp,\ee) = (\x - (\pp - \ee)) / \ee;
        }]
        \pgfmathsetmacro{\eeps}{0.1}
        \pgfmathsetmacro{\xa}{0.33}
        \pgfmathsetmacro{\xb}{0.33 + \eeps}
        \pgfmathsetmacro{\xc}{0.66 + \eeps}
        \pgfmathsetmacro{\xd}{0.66 + 2 * \eeps}
        \pgfmathsetmacro{\xe}{1.00 + 2 * \eeps}
        \pgfmathsetmacro{\ya}{-1.1}
        \pgfmathsetmacro{\yb}{1.1}
        \begin{groupplot}[
            group style={group size=2 by 1, horizontal sep=0.8in},
            width=2.5in,
            height = 2.0in,
            xmin = 0.0, xmax = \xe,
            ymin = \ya, ymax = \yb,
            ticks = none,
            x axis line style={
                draw=none,
                insert path={
                    (axis cs:0,\ya) edge (axis cs:\xa,\ya)
                    (axis cs:\xa,\ya) edge[draw, dotted] (axis cs:\xb,\ya)
                    (axis cs:\xb,\ya) edge (axis cs:\xc,\ya)
                    (axis cs:\xc,\ya) edge[draw, dotted] (axis cs:\xd,\ya)
                    (axis cs:\xd,\ya) edge (axis cs:\xe,\ya)
                    (axis cs:0,\yb) edge (axis cs:\xa,\yb)
                    (axis cs:\xa,\yb) edge[draw, dotted] (axis cs:\xb,\yb)
                    (axis cs:\xb,\yb) edge (axis cs:\xc,\yb)
                    (axis cs:\xc,\yb) edge[draw, dotted] (axis cs:\xd,\yb)
                    (axis cs:\xd,\yb) edge (axis cs:\xe,\yb)
                    (axis cs:0,\ya) edge (axis cs:0,\yb)
                    (axis cs:\xe,\ya) edge (axis cs:\xe,\yb)
                }
            }
            ]
            \nextgroupplot[]
            \addplot[red, thick, domain=0:0.33] {0};
            \addplot[red, thick, domain=0.33:{0.33+\eeps},samples=31] {sin(2*pi*deg(xeeps(x,0.33,\eeps)))};
            \addplot[red, thick, domain={0.33+\eeps}:{0.66+\eeps}] {0};
            \addplot[red, thick, domain={0.66+\eeps}:{0.66+2*\eeps},samples=63] {sin(4*pi*deg(xeeps(x,{0.66+\eeps},\eeps)))};
            \addplot[red, thick, domain={0.66+2*\eeps}:{1.0+2*\eeps}] {0};
            \node[above] at (axis cs: {1.0+2*\eeps-0.15},0.25) {$\phi$};
            \nextgroupplot[]
            \addplot[red, thick, domain=0:0.33] {0};
            \addplot[red, thick, domain=0.33:{0.33+\eeps},samples=63] {sin(4*pi*deg(xeeps(x,0.33,\eeps)))};
            \addplot[red, thick, domain={0.33+\eeps}:{0.66+\eeps}] {0};
            \addplot[red, thick, domain={0.66+\eeps}:{0.66+2*\eeps},samples=31] {sin(2*pi*deg(xeeps(x,{0.66+\eeps},\eeps)))};
            \addplot[red, thick, domain={0.66+2*\eeps}:{1.0+2*\eeps}] {0};
            \node[above] at (axis cs: {1.0+2*\eeps-0.15},0.25) {$\psi$};
        \end{groupplot}
    \end{tikzpicture}
    \caption{Paths that are distinct in $\DD$ but equal in $F'$.}
    \label{fig:distinct_paths}
\end{figure}
\end{example}

We end this subsection by recording some elementary properties of $\DD$ which are used in Section~\ref{sec:fdd}. The following result is immediate from the definitions.

\begin{prop} \label{prop:cup}
Let $a<c<b$. Then 
\[
\alpha_{\infty;[a,b]}(\phi_1,\phi_2)\le
\max\{
\alpha_{\infty;[a,c]}(\phi_1,\phi_2),
\alpha_{\infty;(c,b]}(\phi_1,\phi_2)\}
\]
for all $\phi_1,\phi_2\in \bar\DD[a,b]$.\footnote{We can define $\alpha_\infty$ on $\bar\DD(c,b]$ with the obvious modifications.} \qed
\end{prop}

\begin{defn}\label{def:iota}
Define the embedding 
$\iota:D(I)\hookrightarrow \bar\DD(I)$ by 
setting
$
\iota h(t)(s)=h(t)
$.
We call $\iota h$ the \emph{trivial lift} of $h$.
\end{defn}

\begin{rmk} \label{rmk:lift}
Let $h_1,h_2\in D(I)$.
Then 
$\alpha_\infty(\iota h_1,\iota h_2)
=\sigma_\infty(h_1,h_2)\le |h_1-h_2|_\infty$.
\end{rmk}

\begin{prop} \label{prop:E}
Let $\phi\in\bar\DD(I)$ such that
$\phi(t)(s)\equiv\phi(a)(0)$ for all $t\in[a,b)$, $s\in[0,1]$,
and such that $\phi(b)\in D[0,1]$.
Define 
$h\in D[a,b]$,
$h(s)=\phi(b)(\tfrac{s-a}{b-a})$.
Then
$\alpha_\infty(\phi,\iota h)\le b-a$.
\end{prop}

\begin{proof}
We take $\Pi=\{b\}$.
Note that 
    $(\iota h)^\delta$ is a reparametrisation of $\phi^\delta$ on $[a,b+\delta]$ so
\[
        \sigma_{\infty} (\phi^\delta, (\iota h)^\delta)
\le b-a+\delta.
\]
The result follows by definition of $\alpha_\infty$.
\end{proof}

\begin{rmk}\label{rem:M1}
There is another embedding $\jmath:D[a,b]\hookrightarrow\DD[a,b]$ for which $(\jmath h)(t)(s) = (1-s)h(t^-) + sh(t)$, the linear path from $h(t^-)$ to $h(t)$.
Then the topology on $D[a,b]$ induced by $\alpha_\infty$ via the embedding $\jmath$ is the $\cM_1$ topology, see~\cite[Proposition 2.9]{CF19}.
Together with Remark~\ref{rmk:lift},
this shows that $(\DD[a,b],\alpha_\infty)$ generalises simultaneously the $\cJ_1$ and the $\cM_1$ 
topologies on $D[a,b]$ via the embeddings $\iota,\jmath : D[a,b]\hookrightarrow \DD[a,b]$ respectively.
\end{rmk}

\subsection{\texorpdfstring{$p$}{p}-variation}
\label{sec:pvar}

Fix $p\in[1,\infty]$.
The $p$-variation semi-norm of a function $h:[a,b]\to\R^d$ is given by
\begin{equation*}
|h|_{p\var;[a,b]} = \sup_{a=t_0<t_1<\dots<t_k=b}\Big(\sum_{j=1}^k |h(t_j)-h(t_{j-1})|^p\Big)^{1/p},
\end{equation*}
understood as $\sup_{s,t\in[a,b]} |h_s-h_t|$ if $p=\infty$.
Define also the norm
\[
\|h\|_{p\var} = |h(a)|+|h|_{p\var}.
\]

Now let $\phi\in\bar\DD[a,b]$.
Recalling the definition of $\phi^\delta\in D[a,b+\delta]$,
we set $|\phi|_{p\var}=|\phi^\delta|_{p\var}$. This is well-defined, independent of $\delta$, and constant on equivalence classes by Lemma~\ref{lem:p-var_indep}.
We can therefore define
$|\phi|_{p\var}$ for $\phi\in\DD[a,b]$.
Let 
\begin{equation}\label{eq:DD_pvar}
\DD^{p\var}[a,b]=\{\phi\in\DD[a,b]:
|\phi|_{p\var}<\infty\}. 
\end{equation}
Similarly, define $\bar\DD^{p\var}[a,b]$.

\begin{rmk} \label{rmk:lift2}
Let $h\in D[a,b]$ with trivial lift
$\iota h\in \DD[a,b]$.
Then $|\iota h|_{p\var}=|h|_{p\var}$.
\end{rmk}

Next, we introduce a Skorokhod version of $p$-variation.
For $h_1,h_2 \in D[a,b]$, define the metric
\[
\sigma_{p\var}(h_1,h_2) = 
\inf_{\rho \in \cR}\max \big\{ |\rho-\id|_\infty, \|h_1\circ\rho-h_2\|_{p\var}\big\}.
\]

For $\phi_1,\phi_2\in\bar\DD^{p\var}[a,b]$,
we set
\[\alpha_{p\var}(\phi_1,\phi_2)=\lim_{\delta\to0}\sigma_{p\var}(\phi_1^\delta,\phi_2^\delta).
\]
Again, by Lemma~\ref{lem:limExists}\ref{pt:alpha_p-var} this limit exists and is constant on equivalence classes.
In this way, we define the metric $\alpha_{p\var}$ on $\DD^{p\var}[a,b]$.

\subsection{Differential equations driven by decorated paths}
\label{sec:DEs}

Fix $m,d\ge 1$ and $p\in [1,2)$.
For $\gamma>0$, let $C^\gamma(\R^m,\R^d)$ denote the space of 
functions $f : \R^m \to \R^d$ such that
\[
    \|f\|_{C^\gamma}
    = \max_{|\alpha|=0,\ldots,[\gamma]} \|D^\alpha f\|_\infty
    + \sup_{x\neq y\in \R^m} \max_{|\alpha|=[\gamma]}
    \frac{|D^\alpha f(x) - D^\alpha f(y)|}{|x-y|^{\gamma-[\gamma]}}
    < \infty.
\]
(Note $f\in C^N$ for $N \in \Z^+$  implies only that the $(N-1)$-th derivative of $f$ is Lipschitz.)
Assume that $A\in C^{\beta}(\R^m,\R^m)$ for $\beta>1$ and $B\in C^\gamma(\R^m,\R^{m\times d})$ for $\gamma>p$.
Let $\xi\in\R^m$. We consider initial value problems of the form
\begin{equation} \label{eq:IVP}
d X = A(X)\,dt+B(X)\,d \phi, \quad X(0)=\xi,
\end{equation}
where $\phi\in\DD^{p\var}([0,1],\R^d)$, $X\in\DD^{p\var}([0,1],\R^m)$.

To interpret and solve~\eqref{eq:IVP}, consider the decorated path $(\omega,\Pi) = ((\iota\id,\phi),\Pi)\in \bar\DD^{p\var}([0,1],\R^{1+d})$,
where $\id : [0,1]\to[0,1]$ is the identity map.
Consider also a $\delta$-extension
$\omega^\delta\in D^{p\var}([0,1+\delta],\R^{1+d})$, $\delta>0$, as described in Section~\ref{sec:DD}.
There is a unique solution
$X^\delta\in D^{p\var}([0,1+\delta],\R^m)$ to the initial value problem
\[
dX^\delta=(A,B)(X^\delta)\, d\omega^\delta, \quad X^\delta(0)=\xi
\]
defined as a Young ODE, see Section~\ref{sec:RDEs}.
Define $(X,\Pi)\in\bar\DD^{p\var}([0,1],\R^m)$ with
$X(t_j)\in D^{p\var}([0,1],\R^m)$ for $t_j\in \Pi$ given by
the suitably rescaled version of $X^\delta$ restricted to the interval $I_j=[c_j,c_j+2^{-j}\delta/r]$
(see Section~\ref{sec:RDEs} for details).
By construction, $X^\delta$ is precisely the $\delta$-extension of $(X,\Pi)$.

As shown in Proposition~\ref{prop:stability_ODE}, the resulting map
\[
\cS : \DD^{p\var}([0,1],\R^d) \to \DD^{p\var}([0,1],\R^m),
\qquad \phi\mapsto X,
\]
is well-defined and continuous (with respect to the $\alpha_{p\var}$ topology on domain and range).
We call $X$ the \emph{solution} of~\eqref{eq:IVP} and
$\cS$ the \emph{solution map}.

The following is our main convergence criterion for ODEs driven by random decorated paths, the proof of which we give at the end of Section~\ref{sec:RDEs}.

\begin{thm} \label{thm:S}
Suppose $\phi_n$ is a sequence of random variables in $\DD^{p\var}([0,1],\R^d)$ and $\phi$ is a random variable in $\DD([0,1],\R^d)$, such that
$\phi_n \to_w \phi$ in the $\alpha_\infty$ topology.
Suppose further that $|\phi_n|_{p\var}$ is tight.

Then $|\phi|_{p\var}<\infty$ a.s.\ and 
the solutions 
$X_n=\cS(\phi_n),\, X=\cS(\phi) \in \DD^{p\var}([0,1],\R^m)$
satisfy
$X_n \to_w X$
in the $\alpha_{q\var}$ topology for all $q>p$.
\end{thm}

\section{Dynamical setup and main results}
\label{sec:setup}

In this section, we state the main result in this paper.

\subsection{Assumptions}
\label{sec:ass}

Let $\Lambda$ be a compact metric space with Borel probability measure $\mu_\Lambda$,
and let $T:\Lambda\to\Lambda$ be an ergodic and mixing measure-preserving dynamical system.
Fix $\Sigma\subset\Lambda$ with $\mu_\Lambda(\Sigma)>0$.  We define
the \emph{first return time}
\[
    R : \Sigma\to\Z^+
    , \qquad
    R(x) = \inf \{n\ge 1: T^n x \in \Sigma\},
\]
and the \emph{first return map}
\[
f:\Sigma\to\Sigma,
\qquad  fx=T^{R(x)}x.
\]
 The map $f:\Sigma\to\Sigma$ preserves the probability measure $\mu=\mu_\Lambda|_\Sigma/\mu_\Lambda(\Sigma)$. 

As described below,
we assume that $T$ falls within the so-called Chernov-Markarian-Zhang framework~\cite{Markarian04,ChernovZhang05}.
Informally this means that 
(i) we have good control on the first return time $R$, and
(ii) $f$ is modelled by a Young tower with exponential tails.

This framework incorporates uniformly expanding/hyperbolic systems such as the intermittent maps~\cite{PomeauManneville80} as analysed in~\cite{CFKM20}, but is flexible enough to include nonuniformly hyperbolic systems such as the billiards with flat cusps~\cite{JungZhang18} from Example~\ref{ex:cusps}.

\paragraph{Assumptions on the first return time.}
Our first assumption is that $R$ is regularly varying with exponent $\alpha\in(1,2)$.
That is
\begin{equation} \label{eq:rvR}
    \mu(R>t)\sim \ell(t) t^{-\alpha}
\quad\text{as $t\to\infty$}
\end{equation}
where $\ell:(0,\infty)\to(0,\infty)$ is a slowly varying function.

Introduce $b_n>0$ such that $b_n^\alpha\sim n\ell(b_n)$ as $n\to\infty$.
We require that for every $s>0$, there exist $\theta_1>1$ and $C>0$ such that
\begin{equation} \label{eq:D}
\mu(R>sb_n\text{ and }R\circ f^j>sb_n)\le Cn^{-\theta_1}
\quad\text{for all $1\le j\le n$.}
\end{equation}
This assumption means that large first returns are not too clustered.

\paragraph{Assumptions on the first return map.}

We suppose that $(f,\Sigma,\mu)$ is ergodic and mixing and modelled by a Young tower with exponential tails~\cite{Young98} as discussed further in Section~\ref{sec:pre}.
In particular, 
\begin{itemize}
\item There is an inducing set $Y\subset \Sigma$ with $\mu(Y)>0$ and a return time $\tau:Y\to\Z^+$ (not necessarily the first return time) such that $F=f^\tau$ maps $Y$ into $Y$;
\item  The tail probabilities $\mu(\tau>n)$ decay exponentially;
\item Associated to $F$, there is a countable partition $\{Y_j:j\ge1\}$ of $Y$ and $\tau$ is constant on elements $Y_j$.
\end{itemize}
We require that $R$ is constant on elements of
\begin{equation}\label{eq:C_def}
\cC=\{f^\ell Y_j:0\le \ell < \tau|_{Y_j},\,j\ge1\}.    
\end{equation}

\paragraph{Assumptions on excursions.}

Let $\Sigma_i$, $i=1,\dots,\MM$, be a finite collection of disjoint subsets of $\Sigma$
such that the $\Sigma_i$ are unions of elements of $\cC$.
(When $T : \Lambda \to \Lambda$ is a billiard with flat cusps, the sets $\Sigma_i$ correspond to trajectories
entering the $i$'th cusp, restricting to flattest cusps.)

We suppose that
\begin{equation}
    \label{eq:rv}
    \mu(R \bone_{\Sigma_i}>t)\sim c_i\ell(t) t^{-\alpha}
\quad\text{as $t\to\infty$}
\end{equation}
for each $i$,
where $c_i>0$ and $\sum_{i=1}^\MM c_i=1$.
Setting $\Sigma_0=\Sigma\setminus\bigcup_{i=1,\dots,\MM}\Sigma_i$, we require that there exist  $C>0$ and $\eta>0$ such that
\begin{equation}
    \label{eq:rv0}
    \mu(R \bone_{\Sigma_0}>t)\le Ct^{-(\alpha+\eta)}
\quad\text{for all $t>0$.}
\end{equation}
(Together,~\eqref{eq:rv} and~\eqref{eq:rv0} imply assumption~\eqref{eq:rvR}.)

Let $v:\Lambda\to\R^d$ be a H\"older observable with $\int_\Lambda v\,d\mu_\Lambda=0$.
We consider a set of \emph{profiles} 
$P=\{P_i,\, i=1,\dots,\MM\}\subset
C( [0,1] , \R^d)$ with $P_i(0)=0$ for each $i$.
Our final assumption is that there exist $C>0$ and $\eta\in(0,1)$ such that
\begin{equation} \label{eq:P}
    \bone_{\Sigma_i}\Big|\sum_{\ell=0}^{k}v\circ T^\ell -P_i(k/R)R\Big|
    \le CR^{1-\eta}
\quad\text{for all $0\le k\le R-1$, $i=1,\dots,\MM$.}
\end{equation}

\begin{rmk} By~\eqref{eq:P} and continuity of $P_i$, each profile $P_i$ is Lipschitz. It is equivalent to~\eqref{eq:P} that
\begin{equation} \label{eq:PP}
    \bone_{\Sigma_i}\Big|\sum_{\ell=0}^{[Rt]-1}v\circ T^\ell -P_i(t)R\Big|
    \le CR^{1-\eta}
\quad\text{for all $t\in[0,1]$, $i=1,\dots,\MM$.}
\end{equation}
\end{rmk}

\begin{rmk} Unlike in~\cite{FFMT}, we do not require that the endpoints $P_i(1)\in\R^d$ are distinct and nonzero.
\end{rmk}

\subsection{The decorated L\'evy process \texorpdfstring{$L_\alpha^{P}$}{}}
\label{sec:Lalpha}

Let $\omega_1 = (1, 0, \ldots, 0)$, \ldots, $\omega_{\MM} = (0, \ldots, 0, 1)$ be the standard basis of $\R^{\MM}$.
Define
\begin{equation*}
    Z : \Sigma \to \R^{\MM}
    , \qquad
    Z = \sum_{i =1}^\MM \omega_i R \bone_{\Sigma_i}.
\end{equation*}
Then $Z$ is regularly varying with spectral measure
$\nu$ on $\bbS^{\MM-1}$ given by
$\nu = \sum_{i=1}^\MM c_i \delta_{\omega_i}$.
This means that
\begin{equation} \label{eq:regvar}
        \lim_{t \to \infty}
        \frac{\mu(|Z| > r t, \ Z / |Z| \in E)}{\mu(|Z| > t)}
        = r^{-\alpha} \nu(E)
\end{equation}
    for all $r > 0$ and all Borel sets $E \subset \bbS^{\MM-1}$ with $\nu(\partial E) = 0$.

Let 
$G_\alpha$ be the corresponding $\MM$-dimensional $\alpha$-stable law with
characteristic function
\[
    \E e^{is\cdot G_\alpha}
    = \exp\Big\{
        -\int_{\bbS^{\MM-1}}|s\cdot x|^\alpha\Big(1-i\sgn(s\cdot x)
        \tan\frac{\pi\alpha}{2}\Big)\cos\frac{\pi\alpha}{2}\,\Gamma(1-\alpha)\,d\nu(x)
    \Big\}
\]
for $s\in\R^{\MM}$.
Then $Z$ is in the domain of attraction of $G_\alpha$.
That is, if $Z_1,Z_2,\dots$ are i.i.d.\ copies of $Z$, then $b_n^{-1} \big(\sum_{j=1}^n Z_j-n\int_\Sigma Z\,d\mu\big)\to_w G_\alpha$.

Let $\bar R=\int_\Sigma R\,d\mu$ and
let $L_\alpha\in D([0,1],\R^e)$ be the $\MM$-dimensional $\alpha$-stable L\'evy process corresponding to the stable law $\bar R^{-1/\alpha} G_\alpha$.

Let 
\(
    \Gamma = (P_1(1), \ldots, P_{\MM}(1))\in \R^{d \times \MM}
    ,
\)
and choose $\Phi : \bbS^{\MM-1} \to \Lip([0,1], \R^d)$ such that
\begin{itemize}
\item $(x,s)\mapsto\Phi(x)(s)$ is Lipschitz continuous;
\item $\Phi$ is constant in a neighbourhood of each $\omega_i$;
\item $\Phi(\omega_i) = P_i$ for each $i$.
\end{itemize}
     We define the decorated L\'evy process $L_\alpha^{P} \in \DD([0,1], \R^d)$:
     \[
         L_\alpha^{P} (t) (s)
         = \Gamma L_\alpha(t^-) + |L_\alpha(t) - L_\alpha(t^-)| \Phi \biggl( \frac{L_\alpha(t) - L_\alpha(t^-)}{|L_\alpha(t) - L_\alpha(t^-)|} \biggr) (s)
         ,
     \]
with the convention that $L_\alpha^{P}(t)\equiv \Gamma L_\alpha(t)$ at continuity points $t$ of $L_\alpha$.
Although there is some flexibility in the choice of $\Phi$, the decorated process $L_\alpha^P$ depends only on $L_\alpha$ and the set of profiles $P$.

\begin{rmk}
(a) It is easily checked that $L_\alpha^{P}\in \DD([0,1],\R^d)$.
We need to verify that 
$L_\alpha^{P}(t)(1)=\Gamma L_\alpha(t)$
for discontinuity points $t$ of $L_\alpha$.
To this end, $L_\alpha(t) - L_\alpha(t^-)=\omega_i|L_\alpha(t) - L_\alpha(t^-)|$
for some $i$ and $\Phi(\omega_i)(1)=P_i(1)=\Gamma\omega_i$. Hence
\begin{align*}
L_\alpha^{P}(t)(1) 
& =\Gamma L_\alpha(t^-)+|L_\alpha(t) - L_\alpha(t^-)|\Phi(\omega_i)(1)
\\
& =\Gamma L_\alpha(t^-)+|L_\alpha(t) - L_\alpha(t^-)|\Gamma\omega_i
\\
& =\Gamma L_\alpha(t^-)+\Gamma(L_\alpha(t) - L_\alpha(t^-))
 =\Gamma L_\alpha(t).
\end{align*}
(b) Each $\delta$-extension $(L_\alpha^{P})^\delta$ is an element
of $C([0,1+\delta],\R^d)$.
\end{rmk}

\subsection{Main results}

Let $v:\Lambda\to\R^d$ be a H\"older observable with $\int_\Lambda v\,d\mu_\Lambda=0$.
Define the c\`adl\`ag processes on $(\Lambda,\mu_\Lambda)$,
\[
W_n \in D([0,1], \R^d), \qquad
    W_n(t) = b_n^{-1} \sum_{j=0}^{[nt]-1} v \circ T^j
    .
\]
Using the embedding $\iota:D([a,b],\R^d)\hookrightarrow \DD([a,b],\R^d)$
from Definition~\ref{def:iota}, we obtain the trivial lift 
\[
\hW_n=\iota W_n\in\DD([0,1],\R^d),
\qquad
\hW_n(t)(s)=W_n(t).
\]
We prove:

\begin{thm} \label{thm:Wconv}
$\hW_n\to_{\mu_\Lambda} L_\alpha^{P}$
in the $\alpha_\infty$ topology as $n\to\infty$.
\end{thm}

\begin{thm}
    \label{thm:tight}
    The sequence $|W_n|_{p\var}$ is tight 
on $(\Lambda,\mu_\Lambda)$ for every $p > \alpha$.
\end{thm}

The proofs are postponed to Sections~\ref{sec:pre}, \ref{sec:Z}, \ref{sec:fdd} and~\ref{sec:tight}.
As a consequence, we obtain:

\begin{thm}
    \label{thm:X}
Let $d,m\ge1$, $\xi\in\R^m$.
Assume that $A:\R^m\to\R^m$ and 
$B:\R^m\to\R^{m\times d}$ are sufficiently smooth as in Section~\ref{sec:DEs}.
Let $X_n,\,X\in\DD([0,1],\R^m)$ be solutions for
    \[
dX_n = A(X_n) \, dt + B(X_n) \, d\hW_n, \qquad
    dX = A(X) \, dt + B(X) \, dL_\alpha^{P}
\]
with $X_n(0)=X(0)=\xi$.
    Then \[
X_n\to_{\mu_\Lambda}X
\quad\text{in the $\alpha_{q\var}$ topology as $n\to\infty$}
\]
    for every $q > \alpha$.
\end{thm}

\begin{proof} By Theorem~\ref{thm:tight} and Remark~\ref{rmk:lift2},
$|\hW_n|_{p\var}$ is tight for $p>\alpha$. Hence it follows from
Theorems~\ref{thm:Wconv} and~\ref{thm:S} that
$X_n=\cS(\hW_n)\to_{\mu_\Lambda}\cS(L_\alpha^{P})=X$.
\end{proof}

We say that the profile $P_i$ is \emph{linear} if $\{P_i(s):0\le s\le 1\}$ is contained in a line.
Under the extra assumption that excursions are linear, i.e.\ each profile $P_i$ is linear, we can
identify our solution with a Marcus SDE as follows. 

\begin{thm}\label{thm:Marcus}
In the setting of Theorem~\ref{thm:X},
assume additionally that $P_i$ is linear for each $i=1,\dots,\MM$.
Then the path $t\mapsto X(t)(1)$ solves the Marcus SDE
$dY = A(Y) \, dt + B(Y) \diamond dL_\alpha$.
\end{thm}

\begin{proof}
To show the first claim,
consider the decorated process $L^\ell_\alpha\in\DD([0,1],\R^d)$ defined by
$L^\ell_\alpha(t)(s)=(1-s)L_\alpha(t^-) + sL_{\alpha}(t)$
and the solution
$X^\ell\in\DD([0,1],\R^m)$ to
\[
dX^\ell = A(X^\ell) \, dt + B(X^\ell) \, dL^\ell_\alpha, \quad X(0)=\xi.
\]
By~\cite[Proposition~4.16]{CF19},\footnote{\cite[Proposition~4.16]{CF19} assumes
that the functions $A,B$ are $C^\gamma$ for $\gamma>2$, but the same (even simpler) proof applies in our
setting since we are in the Young regime $p\in[1,2)$.}
the path
$t\mapsto X^\ell(t)(1)$ is almost surely equal to the solution $Y$ of the stated Marcus SDE.
On the other hand, for any $i=1,\ldots, \MM$ and $\lambda\in\R$, since $P_i$ takes values on a line through $P_i(0)=0$,
one has $Z(1)=Z^\ell(1)$ for the solutions $Z,Z^\ell:[0,1]\to \R^m$ to the ODEs $dZ = B(Z)\lambda \, dP_i$
and $dZ^\ell = B(Z^\ell) \lambda P_i(1)\, ds$ with same initial value  $Z(0)=Z^\ell(0)$.
Since every excursion of $L^{P}_\alpha$ is of the form $s\mapsto \Gamma L_\alpha(t-) + \lambda P_i(s)$ for some $i=1,\ldots,\MM$ and $\lambda\in\R$, it readily follows from the construction of the solutions $X,X^\ell$ in Section~\ref{sec:DEs} that $X(t)(1) = X^\ell(t)(1)$.
\end{proof}

\section{Preliminaries about Gibbs-Markov maps and Young towers}
\label{sec:pre}

In this section, we elaborate on the structure of the dynamical systems
$(f,\Sigma,\mu)$ and $(F,Y,\mu_Y)$  from Section~\ref{sec:ass} and we collect together some ergodic-theoretic properties of such dynamical systems for easy reference during 
Sections~\ref{sec:Z}, \ref{sec:fdd} and~\ref{sec:tight}.

Let $(\bY,\mu_\bY)$ be a probability space with an at most countable measurable partition $\{\bY_j:\,j\ge1\}$ and let
$\bF:\bY\to\bY$ be an ergodic measure-preserving map.
Define the separation time $s(y,y')$ to be the least integer $n\ge0$ such that $\bF^n y$ and $\bF^ny'$ lie in distinct partition elements.
We assume that $s(y,y')=\infty$ if and only if $y=y'$; then $d_\theta(y,y')=\theta^{s(y,y')}$ is a metric for $\theta\in(0,1)$. 

We say that
$\bF:\bY\to \bY$ is a \emph{(full-branch) Gibbs-Markov map} if
(i) $\bF:\bY_j\to \bY$ is a measure-theoretic bijection for each $j\ge1$;
(ii) there exists $\theta\in(0,1)$ such that 
$\log\xi$ is $d_\theta$-Lipschitz, where
$\xi = d\mu_\bY/d\mu_\bY\circ \bF$.
For standard facts about Gibbs-Markov maps, we refer to~\cite{AD01,ADU93}.

We say that a function $g$ defined on $\bY$ is \emph{piecewise constant} if
$g$ is constant on partition elements.

\begin{prop} \label{prop:GM}
Let $g\in L^2(\bY)$ be piecewise constant with $\int_\bY g\,d\mu_\bY=0$.
Then 
\[
\biggr|\max_{1\le k\le n}\Big|\sum_{j=0}^{k-1}g\circ \bF^j \Big|\biggr|_{L^2(\bY)} \le C n^{1/2}|g|_{L^2(\bY)}
\]
where $C>0$ is a constant independent of $g$ and $n$.
\end{prop}

\begin{proof}
We suppress the bars. Also, $|\cdot|_p$ denotes $|\cdot|_{L^p(Y)}$.
Given $v:Y\to\R$ continuous, we define $\|v\|_\theta=\|v\|_\infty+\Lip_\theta v$ where
$\Lip_\theta v=\sup_{y\neq y'}|v(y)-v(y')|/d_\theta(y,y')$.
In particular, we can fix $\theta\in(0,1)$ so that
$\Lip_\theta\log\xi<\infty$.

Define the transfer operator $P:L^1(Y)\to L^1(Y)$ (so
$\int_Y Pv\,w\,d\mu_Y=\int_Y v\,w\circ F\,d\mu_Y$ for $v\in L^1(Y)$, $w\in L^\infty(Y)$). 
Then $(Pv)(y)=\sum_j {\xi(y_j)}v(y_j)$ where $y_j$ is the unique preimage of $y$ under $F|_{Y_j}$. 
There exists $C>0$ such that 
${\xi(y_j)}\le C\mu(Y_j)$ and $|{\xi(y_j)}-{\xi(y'_j)}|\le C\mu(Y_j)d_\theta(y,y')$
for all $y,y'\in Y_j$, $j\ge1$.
It follows easily that
$\|Pg\|_\theta\ll |g|_{1}\le |g|_{2}$.

There exist constants $\gamma\in(0,1)$, $C>0$ such that
$\|P^nv\|_\theta\le C \gamma^n\|v\|_\theta$ for all continuous $v:Y\to\R$ with $\int_Y v\,d\mu_Y=0$ and all $n\ge1$.
Define $\chi=\sum_{n=1}^\infty P^ng$. 
Then $|\chi|_\infty\le \sum_{n=1}^\infty \|P^{n-1}Pg\|_\theta
\ll \|Pg\|_\theta\ll |g|_{2}$. 

    Now define the ``martingale-coboundary decomposition''
\[
g=m+\chi\circ F-\chi
\]
where 
$|m|_{2}\ll |g|_{2}$. Since $m\in\ker P$, 
it follows easily that 
$\big| \sum_{j < n} m \circ F^j \big|_{2}= n^{1/2}|m|_{2}$.
Moreover,
$\{m\circ F^j,\,j\ge0\}$ is a sequence 
of reverse martingale differences so, by
Doob's inequality,
$\big| \max_{k \le n} |\sum_{j < k} m \circ F^j| \big|_{2}\le 2n^{1/2}|m|_{2}$.  Hence
    \[
        \Big| \max_{k \le n} \big|\sum_{j < k} g \circ F^j\big| \Big|_{2}
        \le 2n^{1/2}|m|_{2}+2|\chi|_{\infty}
        \ll n^{1/2}|g|_{2}
    \]
as required.
\end{proof}

We require that the map $(F,Y,\mu_Y)$ from Section~\ref{sec:ass}
quotients to 
a Gibbs-Markov map $(\bF,\bY,\mu_\bY)$. Specifically,
there is
a measure-preserving semiconjugacy $\bar\pi:Y\to\bY$ between $F$ and $\bF$.
We say that a function $g$ defined on $Y$ is \emph{piecewise constant} 
if $g$ is constant on $Y_j=\bar\pi^{-1}\bY_j$ for each $j$.

Suppose as in Section~\ref{sec:ass} that $b_n^\alpha\sim n\ell(b_n)$ where
$\alpha\in(1,2)$ and $\ell$ is slowly varying.

\begin{prop} \label{prop:G}
Suppose that $(F,Y,\mu_Y)$ quotients to a Gibbs-Markov map. Let
$g:Y\to[0,\infty)$ be piecewise constant 
and integrable with $\int_Y g\,d\mu_Y=0$.
Then the following are equivalent:
\begin{itemize}
\item[(a)] $b_n^{-1}\sum_{j=0}^{n-1}g\circ F^j\to_{\mu_Y}G$ where $G$ is a nondegenerate $\alpha$-stable law.
\item[(b)]
$\mu_Y(g>t) \sim c\ell(t)t^{-\alpha}$ as $t\to\infty$ for some $c>0$.
\end{itemize}
\end{prop}

\begin{proof}
Since $\bar\pi:Y\to\bY$ is a measure-preserving semiconjugacy, it suffices to prove this result at the level of $(\bF,\bY,\bar\mu_Y)$.  Suppose that
$\mu_\bY(g>t) \sim c\ell(t)t^{-\alpha}$ as $t\to\infty$ for some $c>0$.
By~\cite{AD01}, $b_n^{-1}\sum_{j=0}^{n-1}g\circ \bF^j\to_{\mu_\bY}G$ where $G$ is a nondegenerate $\alpha$-stable law.
The converse holds by~\cite{Gouezel10b}.
\end{proof}

\paragraph{Exponential Young towers}
Let $(\Sigma,d)$ be a metric space and $f:\Sigma\to\Sigma$ a measurable map.
Let $Y\subset\Sigma$ be a measurable subset,
Let $\tau:Y\to\Z^+$ be a return time (so $f^{\tau(y)}y\in Y$ for $y\in Y$).
Define $F=f^\tau:Y\to Y$
and let $\mu_Y$ be an $F$-invariant ergodic Borel probability measure on $Y$.

We suppose as above that 
$(F,Y,\mu_Y)$ quotients to a Gibbs-Markov map $(\bF,\bY,\mu_\bY)$.
Moreover, we require that $\tau:Y\to\Z^+$ is piecewise constant.
(In particular, $\tau$ is well-defined on $\bY$.)
We suppose further that $\mu_Y(\tau>n)=O(e^{-cn})$ for some $c>0$.

Define the \emph{Young tower} $f_\Delta:\Delta\to\Delta$,
\[
\Delta=\{(y,\ell)\in Y\times\Z:0\le \ell<\tau(y)\},
\qquad
f_\Delta(y,\ell)=\begin{cases} (y,\ell+1) & \text{if } \ell\le \tau(y)-2
\\
(Fy,0) & \text{if } \ell=\tau(y)-1
\end{cases}
\]
with ergodic $f_\Delta$-invariant probability measure $\mu_\Delta=(\mu_Y\times{\rm counting})/\bar\tau$ where
$\bar\tau=\int_Y\tau\,d\mu_Y$.

Define the semiconjugacy 
\[
\pi_\Delta:\Delta\to\Sigma,
\qquad
\pi_\Delta(y,\ell)=f^\ell (y),
\]
between $f_\Delta$ and $f$. Then $\mu=(\pi_\Delta)_*\mu_\Delta$
is an $f$-invariant ergodic probability measure on $\Sigma$.
Under certain further technical assumptions as in Young~\cite{Young98}, we say that $(f,\Sigma,\mu)$ is \emph{modelled by a Young tower with exponential tails}.
(These technical assumptions ensure the validity of Theorem~\ref{thm:expdecay} and Lemma~\ref{lem:MV} below, and are not used further in this paper.)

\vspace{2ex}
For $n\ge1$, define the \emph{lap number} $N_n:\Delta\to\Z^+$ to be the integer satisfying
\begin{equation}
    \label{eq:N_n}
    \sum_{j=0}^{N_n(x)-1}\tau(F^jy)\le n+\ell< \sum_{j=0}^{N_n(x)}\tau(F^jy)
\end{equation}
for $x=(y,\ell)\in \Delta$.

\begin{prop} \label{prop:lap}
There exists $C>0$ such that 
$\big|\max_{1\le k\le n}|N_k-k\bar\tau^{-1}|\big|_{L^1(\Delta)}\le C n^{1/2}$
for all $n\ge1$.
\end{prop}

\begin{proof}
Write $\tau_k=\sum_{j=0}^{k-1}\tau\circ F^j$.
Then, for $x=(y,\ell)\in \Delta$,
\[
\tau_{N_n(x)}(y)-\tau(y)\le
 \sum_{j=0}^{N_n(x)-1}\tau(F^jy)-\ell
\le n\le 
\sum_{j=0}^{N_n(x)}\tau(F^jy)-\ell\le 
\tau_{N_n(x)}(y)+\tau(F^{N_n(x)}y).
\]
Hence,
\[
|n-N_n(x)\bar\tau | 
 \le |\tau_{N_n(x)}(y)-N_n(x)\bar\tau |
+\max\{\tau(y),\tau(F^{N_n(x)}y)\}.
\]
Since $N_n\le n$,
\[
\max_{1\le k\le n}|k\bar\tau^{-1}-N_k(x)| \le \bar\tau^{-1}H_n(y)
\]
where
$H_n(y)=\max_{0\le j\le n}|\tau_j(y)-j\bar\tau|
+\max_{0\le j\le n}\tau(F^jy)$.

Since $\tau\in L^2(Y)$
and $H_n$ is independent of $\ell$,
\[
\big|\max_{1\le k\le n}|k\bar\tau^{-1}-N_k|\big|_{L^1(\Delta)}\le 
\bar\tau^{-1}\int_Y \tau \,\bar\tau^{-1}H_n\,d\mu_Y
\ll |H_n|_{L^2(Y)}.
\]

Applying Proposition~\ref{prop:GM} with $g=\tau-\bar\tau$,
\[
\big|\max_{0\le j\le n}|\tau_j-j\bar\tau|\big|_{L^2(Y)} 
=\big|\max_{0\le j\le n}|\tau_j-j\bar\tau|\big|_{L^2(\bY)} 
\ll n^{1/2}.
\]
Also,
\(
|\max_{0\le j\le n}\tau\circ F^j|_{L^2(Y)}\le n^{1/2}|\tau|_{L^2(Y)}.
\)
Hence $|H_n|_{L^2(Y)}\ll n^{1/2}$ as required.
\end{proof}

Recall the set $\cC$ defined by \eqref{eq:C_def}.
Let $\cA$ consist of unions of elements of $\cC$.
Let $\cB$ consist of subsets $B\subset\Sigma$ that are unions of sets of the form $f^{-j}A$ where $A\in\cA$, $j\ge0$.
Clearly, $\cA\subset\cB$.

\begin{thm} \label{thm:expdecay}
Suppose that $f:\Sigma\to \Sigma$ is mixing and modelled by a Young tower with exponential tails.
Then there exist constants $C>0$, $\gamma\in(0,1)$ such that
\[
\big|\mu\big(A\cap f^{-n}B)-\mu(A)\mu(B)|
\le C\gamma^n
\]
for all $A\in\cA$, $B\in\cB$, $n\ge1$.
\end{thm}

\begin{proof}
In general, the Young towers associated to $f$ are mixing only up to a finite cycle. However, since $f$ is mixing, we can reduce by~\cite[Section~10]{Chernov99} (see also~\cite[Section~4.1]{BMT21}) to the case when $f_\Delta:\Delta\to\Delta$ is mixing. 

The observables $\bone_A$ and $\bone_B$ on $\Sigma$ lift to observables $v=\bone_A\circ\pi_\Delta$, $w=\bone_B\circ\pi_\Delta$ on the tower $\Delta$. Since $A,B\in\cB$, 
it follows from Young~\cite{Young98} that,
\[
\big|\mu\big(A\cap f^{-n}B)-\mu(A)\mu(B)|\ll (|v|_\infty+\Lip_\theta v)\,|w|_\infty\,\gamma^n,
\]
where $\Lip_\theta v=\sup_{(y,\ell)\neq (y',\ell)}|v(y,\ell)-v(y',\ell)|/{d_\theta(y,y')}$.
The result follows since $|v|_\infty\le1$, $\Lip_\theta v\le 1$ and $|w|_\infty\le1$.
\end{proof}

\begin{lemma} \label{lem:MV}
Suppose that $f:\Sigma\to \Sigma$ is mixing and modelled by a Young tower with exponential tails and fix $q>1$.
Let $H:\Sigma\to\R^d$ be such that $H\in L^q$ and $\int_\Sigma H\,d\mu=0$.
Suppose that there exists $C>0$, $\eta\in(0,1)$ such that
$|H(x)-H(x')|\le C R(x)d(x,x')^\eta$ for all
$x,x'$ in the same element of $\cC$.

Define $H^Y=\sum_{\ell=0}^{\tau-1}H\circ f^\ell:Y\to\R^d$.
Then
\[
H^Y=m+\chi\circ F-\chi
\]
where $m,\,\chi\in L^p$ for all $p<q$ and
$\{m\circ F^{n-j}:0\le j\le n\}$ is a martingale difference sequence for each $n\ge1$.
\end{lemma}

\begin{proof}
This is part of the proof of~\cite[Lemma~6.2]{MV20}.
\end{proof}

\begin{cor} \label{cor:MV}
Assume the setup of Lemma~\ref{lem:MV} with $q>\alpha$. Then
\[
b_n^{-1}\max_{k\le n}\big|\sum_{j=0}^{k-1}H\circ f^j\big|\to_\mu 0.
\]
\end{cor}

\begin{proof}
Choose $p\in(\alpha,q)$.
Following the proof of~\cite[Lemma~6.2]{MV20}, we can apply inequalities of Doob and Burkholder to deduce from Lemma~\ref{lem:MV} that
\[
\Big|\max_{k\le n}\big|\sum_{j=0}^{k-1}H^Y\circ F^j\big|\Big|_p
\le \Big|\max_{k\le n}\big|\sum_{j=0}^{k-1}m\circ F^j\big|\Big|_p
+2 \Big|\max_{k\le n}|\chi\circ F^k|\Big|_p
\ll n^{1/p}=o(b_n).
\]
In particular,
$b_n^{-1}\max_{k\le n}\big|\sum_{j=0}^{k-1}H^Y\circ F^j\big|\to_{\mu_Y} 0$.
The result follows by a standard inducing argument (e.g.\ a very simplified special case of~\cite[Theorem~2.2]{MZ15}).
\end{proof}

\begin{rmk} \label{rmk:MV}
Let $(T,\Lambda,\mu_\Lambda)$ and $R$ be as in Section~\ref{sec:ass}.
Suppose that $v:\Lambda\to\R^d$ is a H\"older function with
$\int_\Lambda v\,d\mu_\Lambda=0$ and
let $H=\sum_{\ell=0}^{R-1}v\circ T^\ell$. Then 
$\int_\Sigma H\,d\mu=0$ and $H$ satisfies the estimate on
$H(x)-H(x')$ in Lemma~\ref{lem:MV}.
\end{rmk}

\begin{lemma} \label{lem:induce}
Let $V\in L^1(\Sigma)$ with $\int_\Sigma V\,d\mu=0$. Define the induced observable $V^Y:Y\to\R$, $V^Y=\sum_{j=0}^{\tau-1}V\circ f^j$. Let $G$ be a random variable. Then the following are equivalent:
\begin{itemize}
\item[(a)] $b_n^{-1}\sum_{j=0}^{n-1}V\circ f^j\to_\mu G$.
\item[(b)] $b_n^{-1}\sum_{j=0}^{n-1}V^Y\circ F^j\to_{\mu_Y}\bar\tau^{1/\alpha} G$.
\end{itemize}
\end{lemma}

\begin{proof}
Since $Y$ quotients to a Gibbs-Markov map and $\tau$ is piecewise constant and square integrable, it follows for example from~\cite{AD01} that 
$n^{-1/2}(\sum_{j=0}^{n-1}\tau\circ F^j-n\bar\tau)\to_{\mu_Y}\cN$ where $\cN$ is a possibly degenerate Gaussian. In particular,
$b_n^{-1}(\sum_{j=0}^{n-1}\tau\circ F^j-n\bar\tau)\to_{\mu_Y}0$.
Hence, the result follows from~\cite[Theorem A.1 and Remark A.3]{MV20}.
\end{proof}

\section{Convergence in \texorpdfstring{$\cJ_1$}{J1} for the first return dynamics}
\label{sec:Z}

Let $(f,\Sigma,\mu)$ and $R:\Sigma\to\Z^+$ be as in Section~\ref{sec:setup}.
In particular, the map $f$ is mixing and modelled by a Young tower with exponential tails.
We assume conditions~\eqref{eq:D} and~\eqref{eq:rv} and make use
of Propositions~\ref{prop:GM} and~\ref{prop:lap} and Theorem~\ref{thm:expdecay}.
The underlying system $(T,\Lambda,\mu_\Lambda)$ plays no role in this section.

In Section~\ref{sec:Lalpha}, we defined 
\[
Z:\Sigma\to\R^\MM, \qquad
Z=\sum_{i=1}^\MM\omega_i R \bone_{\Sigma_i}:\Sigma\to\R^\MM.
\]
Recall also $b_n^\alpha\sim n\ell(b_n)$ as $n\to\infty$ from Section~\ref{sec:ass}.
Let $\tZ = Z - \int_\Sigma Z \,d\mu$ and define the sequence of processes
\[
    W_n^Z\in D([0,1],\R^\MM),
\qquad
W_n^Z(t)
    = b_n^{-1}\sum_{j=0}^{[nt]-1}\tZ\circ f^j
    ,
\]
on the probability space $(\Sigma,\mu)$.
Let $\tL_\alpha$ be the $\MM$-dimensional $\alpha$-stable L\'evy process corresponding to the stable law $G_\alpha$ described in Section~\ref{sec:Lalpha}
(so $\tL_\alpha=\bar R^{1/\alpha}L_\alpha$).

In this section, we prove

\begin{thm} \label{thm:Z}
    Assume that $f:\Sigma\to \Sigma$ is mixing and modelled by a Young tower with exponential tails. 
    Suppose that conditions~\eqref{eq:D} and~\eqref{eq:rv} hold.
    Then $W_n^Z \to_\mu \tL_\alpha$ in the $\cJ_1$ topology.
\end{thm}

Our proof of Theorem~\ref{thm:Z}
largely follows the approach in~\cite[Section~4]{JungPeneZhang20}
which was written specifically for billiards with flat cusps in the case $d=1$.

On the probability space $(\Sigma,\mu)$, we define the sequence of random point processes
\(
    \cN_n = \sum_{j=1}^n\delta_{(\frac{j}{n}\,,\,b_n^{-1}\tZ\circ f^{j-1})}
\)
on $(0,\infty) \times(\R^\MM \setminus \{0\})$.
The L\'evy measure $\Pi$ corresponding to the L\'evy process $\tL_\alpha$ is given by
\[
\Pi(B)=\alpha\int_{\bbS^{\MM-1}}\int_0^\infty \bone_B(rx)r^{-\alpha-1}\,dr\,d\nu(x)
\]
where $\nu=\sum_{i=1}^\MM c_i\delta_{\omega_i}$ is the spectral measure from Section~\ref{sec:Lalpha}.
Let $\cN$ be the Poisson point process on $(0,\infty)\times (\R^\MM\setminus\{0\})$ with mean measure $\Leb\times\Pi$.

By~\cite[Theorem~4.1]{TyranKaminska10b} (see also~\cite[Theorem~1.2]{TyranKaminska10} and~\cite[Proposition~4.4]{JungPeneZhang20}), to prove Theorem~\ref{thm:Z} it is enough to verify two conditions:
\begin{description}
\item[{Condition I} (Point process convergence).]
$\cN_n \to_\mu \cN$ as $n \to \infty$
in the space of point measures defined on $(0,\infty) \times (\R^\MM \setminus \{0\})$.
\item[{Condition II} (Vanishing small values).]
For every $\gamma>0$,
\[
    \lim_{\eps \to 0} \limsup_{n \to \infty} \mu \biggl(
    \max_{1\le k\le n} \Bigl| \sum_{j=0}^{k-1}
    \bigl( \tZ \bone_{\{ |\tZ| < b_n\eps\}} \bigr)\circ f^j
    - k \int_\Sigma \tZ \bone_{\{ |\tZ| <b_n \eps\}} \,d\mu
    \Bigr|
    > b_n\gamma
    \biggr)=0.
\]
\end{description}
These conditions are verified in the next two subsections.
We denote by $B_a(c)\subset\R^\MM$ the open ball of radius $a$ centred at $c$.

\subsection{Point process convergence}

In this subsection, we verify Condition~I. We follow~\cite[Section~4.5]{JungPeneZhang20}
using~\cite[Theorem~2.1]{PeneSaussol20}.
It is enough to prove convergence of $\cN_n$ to $\cN$ on 
\[
(0,\infty)\times U, \qquad U= \R^\MM\setminus \overline{B_{a_0}(0)}
\]
 for each fixed $a_0>0$.

Fix $a_0$ and let
\[
    A_n = \{ |\tZ| > a_0 b_n \}
    .
\]
Let $\cW$ be the ring of subsets of $U$ generated by sets of the type
$\{ x \in \R^\MM : a < |x| < a', \ x / |x| \in E \}$, where $a_0 < a < a'$ and
$E\subset\bbS^{\MM-1}$ is open with $\nu(\partial E)=0$.
Note that $\cW$ generates the Borel sigma-algebra on $U$ and that $\Pi(\partial W)=0$ for all $W\in\cW$.

For a collection $\cF$ of measurable subsets of $\Sigma$, define
\[
    \cQ_p (\cF)
    = \sup_{\substack{A \in \cF \\ B \in \sigma(\bigcup_{j \ge p} f^{-j} \cF )}}
    \bigl| \mu(A \cap B) - \mu(A) \mu(B) \bigr|
    .
\]

Let
\(
    J_n  = b_n^{-1} \tZ
    .
\)
By~\cite[Theorem~2.1]{PeneSaussol20}, to prove Condition~I it suffices to prove the following:

\begin{lemma} \phantomsection \label{lem:PS} 
\begin{enumerate}[label=(\alph*)]
    \item $\lim_{n\to\infty}\mu \bigl( J_n^{-1} W | A_n \bigr)=
        \Pi(W|U)$ for all $W\in\cW$.

\item        $\cQ_1(J_n^{-1} \cW_0) = o(\mu(A_n))$
for every finite subset $\cW_0$ of $\cW$.
\end{enumerate}
\end{lemma}

\begin{proof}[Proof of Lemma~\ref{lem:PS}(a)]
Without loss, we may suppose that $W=\{x\in\R^\MM:a<|x|<a',\,x/|x|\in E\}$ where
$a_0<a<a'$ and $E\subset\bbS^{\MM-1}$ is open with $\nu(\partial E)=0$.
Since $Z-\tZ$ is constant, it follows from~\eqref{eq:regvar} that
\begin{align*}
\lim_{n\to\infty} & \mu \bigl( J_n^{-1} W | A_n \bigr)
 =\lim_{n\to\infty} \frac{\mu\big(|\tZ|\in(ab_n,a'b_n),\, \tZ/|\tZ| \in E\big)} {\mu(|\tZ|>a_0 b_n)}
\\ & =\lim_{n\to\infty} \frac{\mu\big(|Z|\in(ab_n,a'b_n),\, Z/|Z| \in E\big)} {\mu(|Z|>a_0 b_n)}
 =a_0^\alpha (a^{-\alpha}-(a')^{-\alpha})\nu(E).
\end{align*}
On the other hand,
by the definition of $\Pi$,
\[
\Pi(W|U)=\frac{\Pi(W)}{\Pi(U)}=\frac{(a^{-\alpha}-(a')^{-\alpha})\nu(E)}{a_0^{-\alpha}},
\]
so the result is proved.
\end{proof}

Next, we prove Lemma~\ref{lem:PS}(b).

\begin{prop} \label{prop:An}
    $\mu(A_n)  \sim a_0^{-\alpha} n^{-1}$ as $n\to\infty$.
\end{prop}

\begin{proof}
Since $Z-\tZ$ is constant, $\mu(A_n) \sim \mu(A_n')$
where $A_n'=\{|Z|>a_0b_n\}$.

By~\eqref{eq:rv}, 
\[
\mu(|Z|>t)=\mu(R \bone_{\bigcup_{i=1,\dots,\MM}\Sigma_i}>t)\sim \ell(t)t^{-\alpha}.
\]
Using that $\ell$ is slowly varying and that $b_n^\alpha\sim n\ell(b_n)$,
\[
\mu(A_n)\sim \mu(A_n')\sim \ell(a_0b_n)(a_0 b_n)^{-\alpha}
\sim a_0^{-\alpha}\ell(a_0 b_n)\ell(b_n)^{-1}n^{-1}\sim a_0^{-\alpha}n^{-1},
\]
as required.
\end{proof}

Let $\theta_1>1$ be as in~\eqref{eq:D} for $s=a_0/2$. Without loss, we can suppose that $\theta_1\in(1,2)$.
Let $\gamma\in(0,1)$ be as in Theorem~\ref{thm:expdecay}.

\begin{prop} \label{prop:Ancond}
    There exists $C>0$ such that
    $\mu ( A_n \cap f^{-k} A_n ) \le Cn^{-\theta_1}$
    for all $k, n \ge 1$.
\end{prop}

\begin{proof} 
By Theorem~\ref{thm:expdecay},
\(
\big|\mu(A_n \cap f^{-k}A_n)-\mu(A_n)^2\big|
\ll \gamma^k 
\)
for all $k\ge1$.
By Proposition~\ref{prop:An},
$\mu(A_n)^2\ll n^{-2}\le n^{-\theta_1}$.
Hence
$\mu(A_n\cap f^{-k}A_n)\ll
 n^{-\theta_1}$
uniformly in $k\ge n$.

Let $D=\{R>a_0b_n/2\}$. Then
\[
    A_n 
    \SMALL
\subset \{2|Z|>a_0b_n\}
\subset \{2R>a_0b_n\}=D.
\]
Hence $A_n \cap f^{-k} A_n \subset D \cap f^{-k}D$ and,
by~\eqref{eq:D}, $\mu(A_n \cap f^{-k} A_n )\ll n^{-\theta_1}$ uniformly in $1\le k\le n$.

Combining the estimates for $k \le n$ and $k \ge n$ yields the desired result. 
\end{proof}

Now fix $\cW_0\subset\cW$ finite, and let $\cQ_{n, p} = \cQ_p(J_n^{-1} \cW_0)$.

\begin{prop} \label{prop:Qn}
    $\cQ_{n, p} \le C \gamma^p $ for all $n,p\ge1$.
\end{prop}

\begin{proof}
Let $A\in J_n^{-1}\cW_0$, $B\in \sigma\big(\bigcup_{j\ge p}f^{-j}J_n^{-1}\cW_0\big)$. Write
$B=f^{-p}B'$ where $B'\in \sigma\big(\bigcup_{j\ge 0}f^{-j}J_n^{-1}\cW_0\big)$.
Then
\(
|\mu(A\cap B)-\mu(A)\mu(B)|\ll \gamma^p
\)
by Theorem~\ref{thm:expdecay}.
\end{proof}

Let $\tau_{A_n}(x) = \min\{ n \ge 1 : f^n x \in A_n \}$ for $x\in \Sigma$.

\begin{lemma}\label{lem:JPZgen}
    \(
        \cQ_{n,1}
        \le \cQ_{n,p+1} 
        + \mu(A_n \cap \{ \tau_{A_n} \le p\} ) + \mu(A_n)\mu(\tau_{A_n} \le p)
    \)
    for all $n,p \ge 1$.
\end{lemma}

\begin{proof}
This is identical to~\cite[Lemma~4.9]{JungPeneZhang20}. We give the short argument for completeness.

    Let $A \in J_n^{-1} \cW_0$ and $B \in \sigma \bigl(\bigcup_{j \ge 1} f^{-j} J_n^{-1} \cW_0 \bigr)$.
Note that $A\subset A_n$ by definition of $\cW$.

    Suppose that $\cW_0 = \{W_1, \ldots, W_K \}$.
    Observe that there exists a function
    $g : (\{0,1\}^K)^{\Z^+} \to \{0,1\}$ such that $\bone_B = g(Y_1,Y_2,\ldots)$, where
    $Y_i = \bigl( \bone_{J_n^{-1}W_1}, \ldots, \bone_{J_n^{-1}W_K} \bigr) \circ f^i$.
    Define $B' \in \sigma \bigl(\bigcup_{j \ge p+1} f^{-j}(J_n^{-1} \cW_0) \bigr)$ by
    \[
        \bone_{B'} = g(0,\ldots, 0, Y_{p+1},Y_{p+2},\ldots)
        .
    \]
    Then 
$|\Cov(\bone_A, \bone_{B'})| \le \cQ_{n,p+1}$. Moreover,
    $|\bone_B - \bone_{B'}| \le \bone_{\{\tau_{A_n}\le p\}}$, so
\begin{align*}
        \bigl| \Cov(\bone_A, \bone_B) - \Cov(\bone_A, \bone_{B'}) \bigr|
         & \le \mu \bigl( A \cap \{\tau_{A_n}\le p\} \bigr)
        + \mu(A) \mu(\tau_{A_n} \le p)
        \\
& \le  \mu(A_n \cap \{ \tau_{A_n} \le p\} ) + \mu(A_n)\mu(\tau_{A_n} \le p)
        .
\end{align*}
The result follows.
\end{proof}

\begin{proof}[Proof of Lemma~\ref{lem:PS}(b)]
    Let $p_n = [ n^{\theta_2} ]$ where
    $0<\theta_2<\theta_1-1<1$.
    Applying Proposition~\ref{prop:Ancond},
    \begin{align*}
        \mu( A_n \cap \{\tau_{A_n} \le p_n\} )
        & = \mu \biggl( A_n\cap  \bigcup_{k=1}^{p_n} f^{-k} A_n \biggr)
        \\ 
	& \le \sum_{k=1}^{p_n} \mu(A_n \cap f^{-k} A_n )
         = O ( p_n n^{-\theta_1} )
        = o(n^{-1})
        .
    \end{align*}
    Using Proposition~\ref{prop:An} and invariance of $\mu$ under $f$,
    \[
        \mu( \tau_{A_n} \le p_n )
        = \mu \biggl( \bigcup_{k=1}^{p_n} f^{-k} A_n \biggr)
        \le \sum_{k=1}^{p_n} \mu(f^{-k} A_n)
        = p_n \mu(A_n)
        = O ( p_n n^{-1} )
        = o(1)
        .
    \]
Hence it follows from 
    Proposition~\ref{prop:An} and Lemma~\ref{lem:JPZgen} that
$\cQ_{n,1}\le \cQ_{n,p_n+1}+o(\mu(A_n))$.
By Propositions~\ref{prop:An} and~\ref{prop:Qn},
$\cQ_{n,p_n+1}=o(n^{-1})=o(\mu(A_n))$, so
    $\cQ_1(J_n^{-1} \cW_0)=\cQ_{n,1} = o(\mu(A_n))$.
\end{proof}

This concludes the verification of Condition~I.

\subsection{Vanishing small values}
\label{sec:vsv}

In this subsection, we verify Condition~II. 
The regularly varying sequence $b_n$ is asymptotically equivalent to a strictly increasing sequence, so for the purposes of proving Theorem~\ref{thm:Z} we may suppose without loss of generality that $b_n$ strictly increases and $b_1=1$.

It is convenient to work at the level of the Young tower $f_\Delta:\Delta\to\Delta$. 
Recall from Section~\ref{sec:pre} that $\pi_\Delta:\Delta\to \Sigma$, $\pi(y,\ell)=f^\ell y$
is a measure-preserving semiconjugacy
from $(f_\Delta,\Delta,\mu_\Delta)$ to $(f,\Sigma,\mu)$.
Let $\Delta_i=\pi_\Delta^{-1}\Sigma_i$, $i=1,\dots,\MM$.
Abusing notation, we denote lifted observables $R\circ\pi_\Delta:\Delta\to\Z^+$,
$\tZ\circ\pi_\Delta:\Delta\to\R^\MM$ simply by $R$, $\tZ$ and so on.

Fix $i=1,\dots,\MM$ and set $R^{(i)}=R\bone_{\Delta_i}$.  Let $\bar Z=\int_\Delta Z\,d\mu_\Delta$ and
\[
\tR_\xi^{(i)}= R^{(i)}\bone_{\{\omega_iR\in B_\xi(\bar Z)\}}- \int_\Delta R^{(i)}\bone_{\{\omega_i R\in B_\xi(\bar Z)\}} \, d\mu_\Delta, \quad \xi>0.
\]

The main step in the proof is:

\begin{lemma} \label{lem:vanishing}
There is a constant $C>0$ such that
    \[
        \limsup_{n\to\infty}b_n^{-1}\biggl| \max_{1\le k \le n}
        \Big|\sum_{j = 0}^{k-1} \tR_{\eps b_n}^{(i)} \circ f_\Delta^j\Big| \biggr|_{L^1(\Delta)}
        \le C\eps^{1-\alpha/2} 
        \quad\text{for all $\eps>0$.}
    \]
\end{lemma}

Most of the remainder of this subsection is concerned with proving Lemma~\ref{lem:vanishing}.
At the end of the subsection, we show that Condition~II follows from the lemma.

We begin with the following consequence of condition~\eqref{eq:D}.
Let $\theta_1>1$ be as in~\eqref{eq:D} with $s=1$.

\begin{prop} \label{prop:D}
There exists $a_1>0$ such that
\[
\mu_\Delta(R>\eta\text{ and }R\circ f^j>\eta)\ll 
\Big(\frac{\eta^\alpha}{\ell(\eta)}\Big)^{-\theta_1}
\quad\text{for all $1\le j\le a_1\eta^\alpha/\ell(\eta)$, $\,\eta>1$.}
\]
\end{prop}

\begin{proof}
Since $b_n$ is strictly increasing, we can write
$b_n=\tilde b(n)$ where
$\tilde b:(0,\infty)\to(0,\infty)$ is strictly increasing
and $\tilde b(x)^\alpha\sim x\ell(\tilde b(x))$ as $x\to\infty$.
Note that $\tilde b^{-1}(x)\sim x^\alpha/\ell(x)$ as $x\to\infty$.
Choose $a_1>0$ such that
$a_1 x^\alpha/\ell(x)\le \frac12 \tilde b^{-1}(x)$ for all $x\ge1$.

Given $\eta>1=b_1$, there exists $n\in\Z^+$ such that 
$b_n<\eta\le b_{n+1}$.
By~\eqref{eq:D},
\begin{align*}
\mu_\Delta(R>\eta\text{ and }R\circ f^j>\eta)
& \le \mu_\Delta(R>b_n\text{ and }R\circ f^j>b_n)
\\ &
\ll n^{-\theta_1}\ll (n+1)^{-\theta_1}\le (\tilde b^{-1}(\eta))^{-\theta_1}
\ll \Big(\frac{\eta^\alpha}{\ell(\eta)}\Big)^{-\theta_1}
\end{align*}
for all $1\le j\le n$.
Finally, if $j\le a_1\eta^\alpha/\ell(\eta)$,
then $j\le \frac12 \tilde b^{-1}(\eta)\le \frac12(n+1)\le n$ so the 
estimate holds for all such $j$.
\end{proof}

From now on, we fix $\theta_1$ and $a_1$ as in Proposition~\ref{prop:D}.

Let $1 \le  \eta \le \xi$. Define 
$H=H_\xi: Y \to [0,\infty)$,
 $s=s_\eta : Y \to \Z$, 
\[
H(y) = \sum_{0\le j \le \tau(y)-1} (R^{(i)}\bone_{\{\omega_i R\in B_\xi(\bar Z)\}})(y,j), \quad
    s(y) = \# \{0\le  j \le \tau(y)-1 : R(y,j) > \eta \}
    .
\]

Note that $H =  H' + H'' + H'''$, where
\begin{align*}
    H'(y)   & = \bone_{\{s(y)   =  1\}} \sum_{j < \tau(y)} (R^{(i)}\bone_{\{R>\eta\}}\bone_{\{\omega_i R\in B_\xi(\bar Z)\}})(y,j) 
    ,    \\
    H''(y) & = \sum_{j < \tau(y)} (R^{(i)} \bone_{\{R \le \eta\}}\bone_{\{\omega_i R\in B_\xi(\bar Z)\}})(y,j)
    , \\
    H'''(y)  & = \bone_{\{s(y) \ge 2\}} \sum_{j < \tau(y)} (R^{(i)}\bone_{\{R>\eta\}}\bone_{\{\omega_i R\in B_\xi(\bar Z)\}})(y,j)
    .
\end{align*}

\begin{prop}
    \label{prop:H}
    Let $\delta>0$. There exist $M,\,C>0$ such that
    \begin{enumerate}[label=(\alph*)]
        \item\label{prop:H:'} 
            $|H'|_{L^2(Y)} \le C \ell(\xi)^{1/2}\xi^{1-\alpha / 2}$,
        \item\label{prop:H:''} $|H''|_{L^2(Y)} \le C \eta^{1 - \alpha / 2 + \delta}$,
        \item\label{prop:H:'''}  $|H'''|_{L^2(Y)} \le C \xi^{1+\delta} \Big(\frac{\eta^\alpha}{\ell(\eta)}\Big)^{-\theta_1/2}$,
    \end{enumerate}
    for all $1 <  \eta \le \xi$ satisfying $M\log \xi \le a_1\eta^\alpha/\ell(\eta)$
and $\xi>|\bar Z|$.
\end{prop}

\begin{proof}
We use throughout  that $R^{(i)}\le R$ for all $i$.

\noindent\ref{prop:H:'}
Observe that $H' \le \xi'$ where 
$\xi'=\xi+|\bar Z|$.
Also for $0\le n\le \xi'$,
\begin{align*}
\mu_Y(H'=n) & \le \mu_Y(y\in Y:R(y,j)=n\text{ for some $0\le j<\tau(y)$}\})
\\
& \le \int_Y \sum_{j=0}^{\tau(y)-1}\bone_{\{y\in Y:R(y,j)=n\}}\,d\mu_Y(y)
=\bar\tau\mu_\Delta(R=n).
\end{align*}
    Hence
\[
|H'|_{L^2(Y)}^2=\sum_{n\le \xi'}n^2\mu_Y(H' = n)
\le \bar\tau\sum_{n\le \xi'}n^2\mu_\Delta(R=n)
    \ll \sum_{n \le \xi'} n \mu_\Delta(R \ge n).
\]
Using~\eqref{eq:rv} and applying Karamata's inequality, we conclude that
\(
   |H'|_{L^2(Y)}^2
    \ll \ell(\xi') \xi'^{2 - \alpha}
    \ll \ell(\xi) \xi^{2 - \alpha}
\)
as required.

\vspace{1ex}
\noindent\ref{prop:H:''}
Let $q>2$, $\eps>0$ and
observe that $|R \bone_{\{R \le \eta\}}|_{L^q(\Delta)}^q \ll \sum_{j \le \eta} j^{q-1} \mu_\Delta(R \ge j) \ll \eta^{q+\eps - \alpha}$.
Let $g_j(y)=(R \bone_{\{R\le \eta\}})\circ f_\Delta^j(y,0)$.
Since $f_\Delta$ is measure-preserving,
\[
\begin{split}
\Big|\sum_{j<k}g_j\Big|_{L^q(Y)}^q & =\int_Y\Big| \sum_{j<k}g_j \Big|^q\,d\mu_Y
 \le 
\int_Y\sum_{\ell=0}^{\tau(y)-1}\Big| \sum_{j<k}(R \bone_{\{R\le \eta\}})\circ f_\Delta^j(y,\ell) \Big|^q\,d\mu_Y(y)
\\ & =\bar\tau \int_\Delta \Big| \sum_{j<k}(R \bone_{\{R\le \eta\}})\circ f_\Delta^j \Big|^q\,d\mu_\Delta
\le \bar\tau k^q|R \bone_{\{R \le \eta\}}|_{L^q(\Delta)}^q \ll k^q \eta^{q+\eps - \alpha}.
\end{split}
\]
Note also that
$H''\le \sum_{j<\tau}g_j$.
Since $\tau$ has exponential tails, there exists $c_0>0$ such that
\begin{align*}
    |H''|_{L^2(Y)}^2 
    & \le \sum_{k=1}^\infty \Bigl| \bone_{\{\tau=k\}} \sum_{j<k}g_j \Bigr|_{L^2(Y)}^2
    \ll \sum_{k=1}^\infty e^{-c_0k} \Bigl| \sum_{j<k}g_j \Bigr|_{L^q(Y)}^2
    \\ 
    & \ll \sum_{k=1}^\infty e^{-c_0k}(k \eta^{(q+\eps-\alpha)/q})^2 \ll \eta^{2(q+\eps-\alpha)/q}.
\end{align*}
The desired estimate follows for $q$ and $\eps$ sufficiently close to $2$ and $0$.

\vspace{1ex}
 \noindent       \ref{prop:H:'''}
Since $\tau$ has exponential tails, there exists $c_1>0$ such that
$k^2\mu_Y (\tau = k) \ll e^{-c_1 k}$.
Set $M = 2\alpha\theta_1  / c_1$.

Suppose that $1< \eta\le \xi$ satisfies 
$M\log\xi\le a_1\eta^{\alpha}/\ell(\eta)$.
For any $y\in\{\tau=k, \, s\ge2\}$, there exist
$0\le j_1<j_2 < k$ such that 
$R(y,j_1)>\eta$ and
$R(y,j_2)>\eta$.
By Proposition~\ref{prop:D},
\begin{align*}
    \mu_Y( \tau = k , \; s \ge 2 )
    & \le \bar\tau \sum_{0\le j_1<j_2 < k}\mu_\Delta(R\circ f_\Delta^{j_1}>\eta
    \text{ and } R\circ f_\Delta^{j_2}>\eta)
    \\
    & \le k \bar\tau \sum_{1\le j< k}\mu_\Delta(R>\eta\text{ and }R\circ f_\Delta^j>\eta)
    \ll k^2 \Big(\frac{\eta^\alpha}{\ell(\eta)}\Big)^{-\theta_1}
\end{align*}
for all $k\le a_1 \eta^{\alpha}/\ell(\eta)$ and hence for 
$k\le M\log \xi$.

Noting that
$H'''(y)\ll \bone_{\{s(y)\ge 2\}}\,\xi\tau(y)$,
\begin{align*}
    |H'''|_{L^2(Y)}^2
    & \ll \xi^2\int_Y \bone_{\{s\ge 2\}}\, \tau^2\,d\mu_Y
    \\
    & \ll \xi^2 \sum_{k \le M \log\xi} k^2 \mu_Y( \tau = k , \; s \ge 2 )
    + \xi^2\sum_{k > M \log\xi} k^2 \mu_Y( \tau = k ).
\end{align*}
Now, 
\begin{align*}
 \sum_{k \le M \log\xi} k^2 \mu_Y( \tau = k , \; s \ge 2 )
& \ll 
\sum_{k \le M \log\xi} k^4 \Big(\frac{\eta^\alpha}{\ell(\eta)}\Big)^{-\theta_1}
\ll (\log\xi)^5 \Big(\frac{\eta^\alpha}{\ell(\eta)}\Big)^{-\theta_1}, \\
 \sum_{k > M \log\xi} k^2 \mu_Y( \tau = k )
& \ll \sum_{k > M \log\xi} e^{-c_1k}
\ll \xi^{-c_1M}\le \eta^{-c_1M}=\eta^{-2\alpha\theta_1}. 
\end{align*}
Hence,
$|H'''|_{L^2(Y)}^2 \ll \xi^2 (\log \xi +1 )^5 \Big(\frac{\eta^\alpha}{\ell(\eta)}\Big)^{-\theta_1}$
and the result follows.
\end{proof}

\begin{cor} \label{cor:H}
There is a constant $C>0$ such that
\[
\limsup_{n\to\infty} b_n^{-1} n^{1/2}|H_{\eps b_n}|_{L^2(Y)}\le C \eps^{1-\alpha/2}
\quad\text{for all $\eps>0$.}
\]
\end{cor}

\begin{proof}
Recalling that $\theta_1>1$, we 
choose $\gamma\in(0,1/\alpha)$ sufficiently close to $1/\alpha$ that $\gamma\alpha\theta_1>1$.
    We apply Proposition~\ref{prop:H} with
    $\xi = \eps b_n$, $\eta = n^{\gamma}$.
(Note that for all $\eps>0$, $|\bar Z|\in\R^\MM$, the constraints $1<\eta\le\xi$, $M\log\xi\le a_1\eta^\alpha/\ell(\alpha)$, $\xi>|\bar Z|$ are satisfied for $n$ sufficiently large.)
Then
\[
|H'|_{L^2(Y)}\ll \eps^{1-\alpha/2}b_n \Big(\frac{\ell(\eps b_n)}{\ell(b_n)}\Big)^{1/2} \big(\ell(b_n) b_n^{-\alpha}\big)^{1/2}
\sim \eps^{1-\alpha/2} b_n n^{-1/2}.
\]
Also, $|H''|_{L^2(Y)}\ll n^{\gamma(1-\alpha/2+\delta)}=o(b_n n^{-1/2})$ for $\delta>0$ sufficiently small.
Finally, $|H'''|_{L^2(Y)}\ll  b_n^{1+\delta}n^{-\gamma\alpha\theta_1/2}\tilde\ell(n)$ where $\tilde\ell$ is slowly varying.
Hence,
$|H'''|_{L^2(Y)}\ll  b_n^{1+\delta}n^{-\gamma\alpha\theta_1/2}n^\delta$.
Since $\gamma\alpha\theta_1>1$, we can shrink $\delta$ if necessary so that
$b_n^\delta n^\delta n^{-\gamma\alpha\theta_1/2}=o(n^{-1/2})$.
Hence, $|H'''|_{L^2(Y)}=o(b_n n^{-1/2})$.
\end{proof}

Define 
$\psi_\xi: \Delta \to [0,\infty)$ by
\(
\psi_\xi(y,\ell) = \sum_{0\le j < \ell} (R^{(i)}\bone_{\{\omega_i R\in B_\xi(\bar Z)\}})(y,j).
\)

\begin{cor} \label{cor:psi}
There exists $C>0$ such that
\[
\limsup_{n\to\infty} b_n^{-1}\big|\max_{k\le n}\psi_{\eps b_n}\circ f_\Delta^k\big|_{L^1(\Delta)} \le C\eps^{1-\alpha/2}
\quad\text{for all $\eps>0$.}
\]
\end{cor}

\begin{proof}
Define $h_\xi:\Delta\to[0,\infty)$ by
$h_\xi(y,\ell)=H_\xi(y)$.
Then
$\max_{k\le n}h_\xi\circ f_\Delta^k(y,\ell)
\le \max_{k\le n}H_\xi(F^ky)$.
Since $\tau$ has exponential tails,
\begin{align*}
\big|\max_{k\le n}h_\xi\circ f_\Delta^k\big|_{L^1(\Delta)}
& \le \bar\tau^{-1}\int_Y\tau \max_{k\le n}H_\xi\circ F^k\,d\mu_Y
\\ &\ll 
\big|\max_{k\le n}H_\xi\circ F^k\big|_{L^2(Y)}
\ll n^{1/2}|H_\xi|_{L^2(Y)}.
\end{align*}
Hence it follows from Corollary~\ref{cor:H} that
$\limsup_{n\to\infty}b_n^{-1}|\max_{k\le n}h_{\eps b_n}\circ f_\Delta^k|_{L^1(\Delta)}\ll \eps^{1-\alpha/2}$.
Since $0\le \psi_{\eps b_n}\le h_{\eps b_n}$, we obtain
 the desired estimate for $\max_{k\le n}\psi_{\eps b_n}\circ f_\Delta^k$.
\end{proof}

Define $Q_{n,\xi}:\Delta\to \R$ by $Q_{n,\xi}=N_n\int_Y H_\xi\,d\mu_Y-n\int_\Delta R^{(i)}\bone_{\{\omega_i R\in B_\xi(\bar Z)\}} \,d\mu_\Delta$ where $N_n$ is the lap number defined in~\eqref{eq:N_n}.

\begin{prop} \label{prop:Q}
There exists $C>0$ such that
\[
\limsup_{n\to\infty} b_n^{-1}\big|\max_{k\le n}|Q_{k,\eps b_n}|\big|_{L^1(\Delta)}
        \le C \eps^{1-\alpha/2}
\quad\text{for all $\eps>0$.}
\]
\end{prop}

\begin{proof}
Observe that
$Q_{n,\xi}
=(N_n-n\bar\tau^{-1})\int_Y H_\xi\,d\mu_Y$. Hence
\[
|Q_{n,\xi}|\le |N_n-n\bar\tau^{-1}|\int_Y H_\xi\,d\mu_Y
\le |N_n-n\bar\tau^{-1}|\,|H_\xi|_{L^2(Y)}.
\]
The result follows from Proposition~\ref{prop:lap} and Corollary~\ref{cor:H}.
\end{proof}

\begin{proof}[Proof of Lemma~\ref{lem:vanishing}]
    Let $\tH_\xi = H_\xi - \int_Y H_\xi \, d\mu_Y$.
Define $G_{n,\xi}:Y\to\R$, $g_{n,\xi}:\Delta\to\R$,
\[
G_{n,\xi}=\max_{k \le n} \sum_{j < k} \tH_\xi \circ F^j,
\quad
g_{n,\xi}(y,\ell)=G_{n,\xi}(y).
\]
By Proposition~\ref{prop:GM} and Corollary~\ref{cor:H},
\[
        \limsup_{n\to\infty}b_n^{-1}| G_{n,\eps b_n}|_{L^2(Y)}
\ll
\limsup_{n\to\infty}b_n^{-1}n^{1/2}|H_{\eps b_n}|_{L^2(Y)}
        \ll \eps^{1-\alpha/2}.
\]
Since $\tau$ has exponential tails,
\[
\int_\Delta |g_{n,\eps b_n}|\,d\mu_\Delta=\bar\tau^{-1}\int_Y \tau|G_{n,\eps b_n}|\,d\mu_Y
\ll | G_{n,\eps b_n}|_{L^2(Y)},
\]
so
    \begin{equation}\label{eq:h}
\limsup_{n\to\infty}b_n^{-1}\int_\Delta |g_{n,\eps b_n}|\,d\mu_\Delta
\ll \eps^{1-\alpha/2}.
    \end{equation}

Next, by the definition of the lap number $N_n$,
\[
\sum_{j=0}^{n-1}(R^{(i)}\bone_{\{\omega_i R\in B_\xi(\bar Z)\}})(f_\Delta^j(y,\ell))=
\sum_{j=0}^{N_n(y,\ell)-1}H_\xi(F^jy)+\psi_\xi(f_\Delta^n(y,\ell))-\psi_\xi(y,\ell),
\]
and so
\[
\sum_{j=0}^{n-1}\tR_\xi^{(i)}(f_\Delta^j(y,\ell))
=\sum_{j=0}^{N_n(y,\ell)-1}\tH_\xi(F^jy)+\psi_{\xi}(f_\Delta^n(y,\ell))-\psi_{\xi}(y,\ell)+Q_{n,\xi}(y,\ell).
\]
Since $N_n\le n$ for all $n$,
\[
    \max_{k\le n}\Big|\sum_{j=0}^{k-1}\tR_\xi^{(i)}\circ f_\Delta^j(y,\ell)\Big|
    \ll \max_{k\le n}\big|\sum_{j<k}\tH_\xi(F^jy)\big|
    + \max_{k\le n}\psi_{\xi}(f_\Delta^k(y,\ell))+\max_{k\le n}|Q_{k,\xi}(y,\ell)|.
\]
In other words,
\[
    \max_{k\le n}\Big|\sum_{j=0}^{k-1}\tR_\xi^{(i)}\circ f_\Delta^j\Big| \ll g_{n,\xi}
    + \max_{k\le n}\psi_{\xi}\circ f_\Delta^k+\max_{k\le n}|Q_{k,\xi}|.
\]
Setting $\xi=\eps b_n$, the result follows from~\eqref{eq:h}, Corollary~\ref{cor:psi} and Proposition~\ref{prop:Q}.
\end{proof}

We can now complete the verification of Condition II.
First notice that
\[
R^{(i)} \bone_{\{|\tZ|<\xi\}}
=R^{(i)} \bone_{\{|\omega_i R^{(i)}-\int_\Delta Z\,d\mu|<\xi\}}
=R^{(i)} \bone_{\{\omega_i R\in B_\xi(\bar Z)\}}.
\]
Hence 
\[
Z\bone_{\{|\tZ|<\xi\}} = \sum_{i} \omega_i R^{(i)} \bone_{\{\omega_i R\in B_\xi(\bar Z)\}}
\]
and so
\[
Z\bone_{\{|\tZ|<\xi\}}-\int_\Delta Z\bone_{\{|\tZ|<\xi\}}d\mu_\Delta = \sum_{i} \omega_i \tR_\xi^{(i)}.
\]
By Lemma~\ref{lem:vanishing},
    \[
        \limsup_{n\to\infty}b_n^{-1}\biggl|\max_{1\le k \le n}\Big| \sum_{j = 0}^{k-1} \big(Z\bone_{\{|\tZ|<\eps b_n\}}\big) \circ f_\Delta^j-k\int_\Delta Z\bone_{\{|\tZ|<\eps b_n\}}d\mu_\Delta\Big| \biggr|_{L^1(\Delta)}\ll \eps^{1-\alpha/2} .
    \]
Also, by Proposition~\ref{prop:GM} with $g=\bone_{\{|\tZ|<\eps b_n\}}-\int_\Delta \bone_{\{|\tZ|<\eps b_n\}}d\mu_\Delta$,
    \[
        \biggl| \max_{1\le k \le n}\Big| \sum_{j = 0}^{k-1} \bone_{\{|\tZ|<\eps b_n\}} \circ f_\Delta^j-k\int_\Delta \bone_{\{|\tZ|<\eps b_n\}}d\mu_\Delta\Big| \biggr|_{L^2(\Delta)}\ll n^{1/2}.
    \]
Combining these two estimates yields
    \[
        \limsup_{n\to\infty}b_n^{-1}\biggl| \max_{1\le k \le n}\Big| \sum_{j = 0}^{k-1} \big(\tZ\bone_{\{|\tZ|<\eps b_n\}}\big) \circ f_\Delta^j-k\int_\Delta \tZ\bone_{\{|\tZ|<\eps b_n\}}d\mu_\Delta\Big| \biggr|_{L^1(\Delta)}\ll \eps^{1-\alpha/2} .
    \]
Hence Condition II follows from Markov's inequality and the fact that $\pi_\Delta:\Delta\to \Sigma$ is a measure-preserving semiconjugacy.
This completes the proof of Theorem~\ref{thm:Z}.

\section{Proof of Theorem~\ref{thm:Wconv}}
\label{sec:fdd}

In this section, we complete the proof of Theorem~\ref{thm:Wconv},
using the results of Section~\ref{sec:pre} (specifically Lemma~\ref{lem:MV} and
Remark~\ref{rmk:MV}) and the $\cJ_1$ convergence from
Section~\ref{sec:Z}.
The structure of the proof broadly follows~\cite{FFMT,MZ15}.

        In Section~\ref{sec:Z}, we defined processes
        \[
W_n^Z\in D([0,1],\R^\MM), \qquad
W_n^Z(t) = b_n^{-1}\sum_{j=0}^{[nt]-1}\tZ\circ f^j,
\]
        where $\tZ = Z - \int_\Sigma Z \,d\mu$. 
By Theorem~\ref{thm:Z},
\(
            W_n^Z \to_\mu \tL_\alpha
\)
            in the $\cJ_1$ topology,
        where $\tL_\alpha$ is an $\MM$-dimensional $\alpha$-stable L\'evy process.
        Recall that $L_\alpha=\bar R^{-1/\alpha}\tL_\alpha$ is the $\alpha$-stable L\'evy process corresponding
        to the stable law $\bar R^{-1/\alpha} G_\alpha$
and that $L_\alpha^P\in\DD([0,1],\R^d)$ is the decorated L\'evy process.

\vspace{1ex}
\noindent\textbf{Step 1:}
        Define processes on $(\Sigma,\mu)$,
        \[
U_n\in D([0,1],\R^\MM), \qquad            
U_n(t) = 
        b_n^{-1} \sum_{j=0}^{N^R_{[nt]}-1} \tZ \circ f^j,
\]
        where $N_k^R$ is the largest integer~$n$ satisfying $\sum_{j=0}^{n-1}R\circ f^j\le k$.
        (This is analogous to, but different from, the lap number $N_n$ defined in Section~\ref{sec:pre}.)
        Then $U_n$ is a time-changed version of $W_n^Z$, and using Theorem~\ref{thm:Z}
        we show that
\[
            U_n \to_\mu L_\alpha
\quad\text{in the $\cJ_1$ topology.}
\]

\vspace{1ex}
\noindent\textbf{Step 2:}
We define decorated processes $G(U_n)\in\DD([0,1], \R^d)$ and show that
\[
            G(U_n) \to_\mu L_\alpha^{P}
\quad\text{in the $\alpha_\infty$ topology.}
\]

\noindent\textbf{Step 3:}
Recall that $\hW_n\in\DD([0,1],\R^d)$ is the trivial lift of $W_n$.
        We show that 
\(
            \alpha_{\infty} (G(U_n), \hW_n) {\to_\mu 0}
            .
\)
From this and Step 2,
        \[
\hW_n \to_\mu L_\alpha^{P}
\quad\text{in the $\alpha_\infty$ topology.}
\]

\noindent\textbf{Step 4:}
        We apply~\cite[Theorem~1]{Zweimuller07} to show that
\(
            \hW_n \to_{\mu_\Lambda} L_\alpha^{P}\)
in the $\alpha_\infty$ topology.

\vspace{1ex}
In the remainder of this section we implement the steps above.  
To carry out Step 1, we apply the argument in~\cite[Lemma~5.7]{FFMT}. 
In this way we obtain the following consequence of
Theorem~\ref{thm:Z}:

\begin{cor}
    \label{cor:U}
    $U_n \to_\mu L_\alpha$ in the $\cJ_1$ topology. \qed
\end{cor}

Turning to Step 2, recall $\Gamma\in \R^{d\times e}$ and the map $\Phi: \bbS^{e-1}\to \Lip([0,1],\R^d)$
from Section~\ref{sec:Lalpha}.
Define
\[
G:D([0,1],\R^\MM)\to\DD([0,1],\R^d), \qquad h\mapsto G(h)
\]
where
     \[
         G(h) (t) (s)
         =
\begin{cases}
\Gamma h(t^-) + |h(t) - h(t^-)| \Phi \bigl( \frac{h(t) - h(t^-)}{|h(t) - h(t^-)|} \bigr) (s)
& \text{if } s\in[0,1) \text{ and } h(t^-)\neq h(t)\\
\Gamma h(t) & \text{if } s=1 \text{ or } h(t^-)=h(t)
\end{cases}
     \]

\begin{lemma}\label{lem:G_cont}
The map $G$ is continuous from $(D([0,1],\R^\MM),\sigma_\infty)$ to $(\DD([0,1],\R^d),\alpha_\infty)$.
\end{lemma}

\begin{proof}
It is clear that $G(h)\in \DD([0,1],\R^d)$ for all $h\in D([0,1],\R^{\MM})$ since $G(h)(\cdot)(1)=\Gamma h$ and $h$ is c{\`a}dl{\`a}g by assumption.

To prove continuity of $G$, let $\eps>0$ and $g,h\in D([0,1],\R^\MM)$ with $\sigma_\infty(g,h)<\eps$.
Choose  $\rho\in\cR$ such that $|g\circ\rho-h|_\infty <\eps$
and $|\rho-\id|_\infty<\eps$.
Let $\Pi=\{t_1,t_2\ldots\}\subset [0,1]$ be the set of
discontinuities of the triple $(g,g\circ\rho,h)$ ordered arbitrarily.
All $\delta$-extensions taken below are with respect to the set $\Pi$.

Let $|\Gamma|$ denote the operator norm of $\Gamma$ and
let $|\Phi|_\infty = \sup_{x\in\bbS^{e-1},s\in[0,1]}|\Phi(x)(s)|$.
Choose $L>0$ such that
$|\Phi(x)(s)-\Phi(y)(s)|\le L|x-y|$.
For $\delta>0$, we show that
\begin{equation}\label{eq:ab_bound}
|G(g\circ\rho)^\delta - G(h)^\delta|_\infty \le |\Gamma|\eps + 8\eps |\Phi|_\infty + 8\eps L.
\end{equation}
Note also that $G(g\circ\rho)^\delta$ is a reparametrisation of $G(g)^\delta$, so
 $\sigma_\infty(G(g)^\delta,G(g\circ\rho)^\delta)\le \eps+\delta$.
Hence
$\sigma_\infty(G(g)^\delta,G(h)^\delta) \le |\Gamma|\eps + 8\eps |\Phi|_\infty + 8\eps L+\eps
+\delta$ and the result follows.

Turning to the verification of~\eqref{eq:ab_bound},
let $\tilde g=g\circ \rho$ and set
$a=G(\tilde g)$, $b=G(h)$.
Let $I_j=[c_j,d_j]\subset [0,1+\delta]$ be the fictitious time intervals appearing in the construction of $a^\delta$ and $b^\delta$ (see Section~\ref{sec:DD}) and denote $J=\bigcup_j [c_j,d_j)$.
Since $|\tilde g-h|_\infty < \eps$, 
\begin{equation}\label{eq:t_not_in_J}
\sup_{t\in [0,1+\delta]\setminus J}|a^\delta(t)-b^\delta(t)|\le \sup_{t\in[0,1]}
|\Gamma(\tilde g(t)-h(t))|\le|\Gamma|\eps.
\end{equation}

Consider now $t_j\in \Pi$.
Let
$\Delta_g=\tilde g(t_j^-)-\tilde g(t_j)$,
$\Delta_h=h(t_j^-)-h(t_j)$ and
$\Delta=\max\{|\Delta_g|,|\Delta_h|\}$.
We have the basic estimate
\begin{equation}\label{eq:naive}
\sup_{t\in [c_j,d_j)} |a^\delta(t)-b^\delta(t)| 
\le |\Gamma(\tilde g(t_j)-h(t_j))|+2\Delta |\Phi|_\infty
\le |\Gamma|\eps +2\Delta |\Phi|_\infty.
\end{equation}
In addition, we note that
$|\Delta_g-\Delta_h|<2\eps$, so if $\Delta>4\eps$ then
\[
\Big|\frac{\Delta_g}{|\Delta_g|} - \frac{\Delta_h}{|\Delta_h|}\Big| 
\le
2|\Delta_g-\Delta_h| /|\Delta_g|
<
 4\eps/(\Delta-2\eps) < 8\eps/\Delta.
\]
Again in the case $\Delta>4\eps$, note that $t_j$ is necessarily a discontinuity of both $\tilde g$ and $h$. Hence
\begin{align} \nonumber
\sup_{t\in [c_j,d_j)}|a^\delta(t)-b^\delta(t)| & \le 
|\Gamma(\tilde g(t_j^-)-h(t_j^-))| + |\Delta_g-\Delta_h||\Phi|_\infty
\\ 
& \qquad+|\Delta_h|{\textstyle\sup_{s\in[0,1]}} |\Phi(\Delta_g/|\Delta_g|)(s)-\Phi(\Delta_h/|\Delta_h|)(s)| 
\nonumber
\\ & \le 
|\Gamma|\eps  + 2\eps |\Phi|_\infty + 8\eps L.
\label{eq:refined}
\end{align}
Applying
\eqref{eq:naive} for $\Delta\le 4\eps$ and \eqref{eq:refined} for $\Delta >4\eps$, we obtain
\[
\sup_{t\in I_j}|a^\delta(t)-b^\delta(t)| \le 
|\Gamma|\eps  + 8\eps |\Phi|_\infty + 8\eps L.
\]
This combined with~\eqref{eq:t_not_in_J} yields
the required estimate~\eqref{eq:ab_bound}
\end{proof}

We now complete Step 2, noting that $L_\alpha^{P}=G(L_\alpha)$.

\begin{cor}
    \label{cor:Udec}
    $G(U_n) \to_\mu L_\alpha^{P}$ in the $\alpha_\infty$ topology.
\end{cor}

\begin{proof}
This follows from  Corollary \ref{cor:U}, Lemma \ref{lem:G_cont} and the continuous mapping theorem.~
\end{proof}

Next, we carry out Step 3. 
Recall that $v:\Lambda\to\R^d$ is H\"older with $\int_\Lambda v\,d\mu_\Lambda=0$.
Define $R':\Sigma\to\Z$ and $V:\Sigma\to\R^d$,
\[
R'=R^{1-\eta}+R\bone_{\Sigma_0}, \qquad V = \sum_{j=0}^{R-1} v \circ T^j,
\]
where $\eta>0$ is as in~\eqref{eq:P}.

\begin{prop} \label{prop:R'}
There exists $p>\alpha$ such that
$R'\in L^p$. Moreover, 
$V=\Gamma Z+H$ where $H\in L^p$.
\end{prop}

\begin{proof}
Since $R\in L^q$ for all $q<\alpha$, it follows that
$R^{1-\eta}\in L^p$ for some $p>\alpha$.
Also, $R\bone_{\Sigma_0}\in L^p$ for some $p>\alpha$ by~\eqref{eq:rv0}.

By~\eqref{eq:P}, 
\[
|H|=|V-\Gamma Z|=
    |V-P_i(1)R| \ll R^{1-\eta}
\]
on $\Sigma_i$, $i=1,\dots,\MM$.
On $\Sigma_0$, we have $|H|=|V|\ll R$.
Hence $|H|\ll R'\in L^p$.
\end{proof}

Note that $\int_\Sigma V\,d\mu=0$.
Hence $V=\Gamma\tZ+\tH$ where
$\tH=H-\int_\Sigma H\,d\mu$.
Define
$\tH_k=\sum_{j=0}^{k-1}\tH\circ f^j$.

    Using that $R$ is the first return time to $\Sigma$, define $R_\Lambda : \Lambda \to \Z$
    so that $R_\Lambda(T^\ell x) = R(x)$ for every $x \in \Sigma$ and $0 \le \ell  < R(x)$.

\begin{lemma} \label{lem:WU}
There exists $C>0$ such that
\[
\alpha_{\infty} (G(U_n), \hW_n)
\le 
\max_{k<n}\{b_n^{-1}|\tH_k|+Cb_n^{-1}R'\circ f^k+n^{-1}R\circ f^k\}
+|v|_\infty b_n^{-1}R_\Lambda\circ T^n
\]
for all $n\ge1$.
\end{lemma}

\begin{proof}
    Set $t_0 = 0$ and $t_{k+1} = t_k + n^{-1} R \circ f^k$, so that the time intervals
    $[t_k, t_{k+1}]$ correspond to excursions from $\Sigma$.
In particular,
$W_n(t_k)=b_n^{-1}\sum_{j=0}^{k-1}V\circ f^j$ and
$U_n(t_k)=b_n^{-1}\sum_{j=0}^{k-1}\tZ\circ f^j$. 
Hence
\begin{equation} \label{eq:WU}
W_n(t_k)-\Gamma U_n(t_k)= 
b_n^{-1}\sum_{j=0}^{k-1}(V-\Gamma\tZ)\circ f^j=
b_n^{-1}\tH_k.
\end{equation}

    Let $\ELL = \max \{ k : t_k \le 1 \}$. Note that $t_k$ and $\ELL$ are random variables depending also on $n$, and that $\ELL\le n$.

    We are going to bound $\alpha_{\infty; (t_k,t_{k+1}]} (G(U_n), \hW_n)$ with $0\le k < \ELL$
    and $\alpha_{\infty; (t_\ELL,1]} (G(U_n), \hW_n)$ separately. 
(The last term is absent if $t_\ELL=1$.) By Proposition~\ref{prop:cup},
    \begin{equation}
        \label{eq:aaa}
        \alpha_{\infty} (G(U_n), \hW_n)
        \le \max \Bigl\{
            \max_{0\le k < \ELL} \alpha_{\infty; (t_k,t_{k+1}]} (G(U_n), \hW_n)
           \ , \ 
            \alpha_{\infty; (t_\ELL,1]} (G(U_n), \hW_n)
        \Bigr\}
        .
    \end{equation}

    Let $0\le k < \ELL$. Define $E \in D([t_k,t_{k+1}], \R^d)$,
    \[
        E(t) = G(U_n)(t_{k+1})(s)
\qquad s= \frac{t-t_k}{t_{k+1}-t_k}=\frac{(t-t_k)n}{R\circ f^k}
        .
    \]
Let $\hE=\iota E\in \DD((t_k,t_{k+1}], \R^d)$ denote the trivial lift
of $E$ from Definition~\ref{def:iota}.
Recall that $U_n$ is constant on $[t_k,t_{k+1})$.
By Proposition~\ref{prop:E},
    \[
        \alpha_{\infty; (t_k, t_{k+1}]} (G(U_n), \hE)
        \le t_{k+1} - t_k
        = n^{-1} R \circ f^k
        .
    \]

Next, we estimate $|W_n-E|_{\infty,(t_k,t_{k+1}]}$.
By~\eqref{eq:WU},
\begin{equation} \label{eq:WE}
W_n(t_k)-E(t_k)= W_n(t_k)-\Gamma U_n(t_k)= b_n^{-1}\tH_k.
\end{equation}

Note that $U_n(t_{k+1})-U_n(t_k)=b_n^{-1}\tZ\circ f^k$.
Hence, for $t\in(t_k,t_{k+1})$, 
\[
E(t) - E(t_k) 
= b_n^{-1}\big\{|\tZ|\, \Phi(\tZ/|\tZ|)(s)\big\}\circ f^k
.
\]
Recall that
$\tZ=Z-\int_\Sigma Z\,d\mu$, $Z=\sum_{i=1}^{\MM} \omega_i R\bone_{\Sigma_i}$.
Since $\Phi:\bbS^{\MM-1}\to \Lip([0,1],\R^d)$ is constant on a neighbourhood of each $\omega_i$ with value $P_i$,
there exists $R_0>0$ such that
\[
\bone_{\Sigma_i}\Phi(\tZ/|\tZ|)=\bone_{\Sigma_i}\Phi(Z/|Z|)=\Phi(\omega_i)=P_i
\]
for $R\ge R_0$ for each $i=1,\dots,\MM$. 
Also, $\bone_{\Sigma_0}|\tZ|\, \Phi(\tZ/|\tZ|)\le |\int_\Sigma Z\,d\mu|_\infty |\Phi|_\infty<\infty$.
Hence,
for $ t\in (t_k,t_{k+1})$,
\begin{align} \label{eq:DE}
E(t)-E(t_k)
=b_n^{-1}\sum_{i=1}^{\MM} \Big\{R\bone_{\Sigma_i}
P_i\big(\tfrac{(t-t_k)n}{R}\big)\Big\}\circ f^k + O(b_n^{-1}).
\end{align}

On the other hand, for $t\in (t_k,t_{k+1}]$,
it follows from~\eqref{eq:PP} that
\begin{align} \label{eq:DW}
W_n(t)   -W_n(t_k) &
 = b_n^{-1}\sum_{\ell=0}^{[(t-t_k)n]-1}v\circ T^\ell\circ f^k
\\ & = b_n^{-1}\Big\{\sum_{i=1}^\MM R\bone_{\Sigma_i}P_i\big(\tfrac{(t-t_k)n}{R}\big) + 
O(R^{1-\eta})+O(R\bone_{\Sigma_0})\Big\}\circ f^k. \nonumber
\end{align}

Combining~\eqref{eq:WE},~\eqref{eq:DE} and~\eqref{eq:DW}, there exists $C>0$ such that
\[
|W_n-E|_{\infty;(t_k,t_{k+1}]}\le b_n^{-1}\big\{\max\{|\tH_k|, |\tH_{k+1}|\}+ CR'\circ f^k\big\}.
\]
Since $\hW_n$ and $\hE$ are trivial lifts,
it follows from Remark~\ref{rmk:lift} that
\[
        \alpha_{\infty;(t_k,t_{k+1}]}(\hW_n,\hE)\le 
|W_n-E|_{\infty;(t_k,t_{k+1}]}\le b_n^{-1}\big\{\max\{|\tH_k|, |\tH_{k+1}|\}+
CR'\circ f^k\big\}.
\]

Combining the estimates for $\alpha_\infty(G(U_n),\hE)$ and $\alpha_\infty(\hW_n,\hE)$, we obtain
    \begin{equation}
        \label{eq:WUt}
        \begin{aligned}
            \alpha_{\infty; (t_k,t_{k+1}]} (G(U_n), \hW_n)
            & \le \alpha_{\infty; (t_k,t_{k+1}]} (G(U_n), \hE) + 
            \alpha_{\infty; (t_k,t_{k+1}]}
(\hW_n , \hE)
            \\
& \le \max_{k\le n}\{b_n^{-1}|\tH_k|+Cb_n^{-1}R'\circ f^k+n^{-1}R\circ f^k\}
            .
        \end{aligned}
    \end{equation}
    Finally, we treat the interval $(t_\ELL, 1]$ for $t_\ELL < 1$.
    On this interval, $G(U_n) = \Gamma U_n(t_\ELL)$ is constant.
    Note that $R \circ f^\ELL = R_\Lambda \circ T^n$ and it follows
that $| W_n(t_\ELL) - W_n |_{\infty;(t_\ELL, 1]}\le b_n^{-1}|v|_\infty  R_\Lambda \circ T^n$. By Remark~\ref{rmk:lift} and~\eqref{eq:WU},
    \begin{equation}
        \label{eq:WUl}
        \begin{aligned}
            \alpha_{\infty; (t_\ELL,1]} (G(U_n), \hW_n)
            & \le |\Gamma U_n(t_\ELL)- W_n|_{\infty; (t_\ELL,1]} 
             \\ & \le |\Gamma U_n(t_\ELL) - W_n(t_\ELL)|
            + | W_n(t_\ELL) - W_n |_{\infty;(t_\ELL, 1]}
            \\
            & \le b_n^{-1}\big\{|\tH_\ELL|
            + |v|_\infty R_\Lambda \circ T^n\big\}
            .
        \end{aligned}
    \end{equation}

    The result follows from~\eqref{eq:aaa}, \eqref{eq:WUt} and~\eqref{eq:WUl}.
\end{proof}

We can now complete Step 3.
\begin{cor} \label{cor:WU}
$\alpha_{\infty} (G(U_n), \hW_n) \to_\mu 0$.
\end{cor}

\begin{proof}
We show that all the terms on the right-hand side in Lemma~\ref{lem:WU}
converge to $0$ in probability on $(\Sigma, \mu)$.

Let $p>\alpha$ be as in Proposition~\ref{prop:R'}.
By Proposition~\ref{prop:R'}, $\tH=V-\Gamma \tZ\in L^p$.
Also $\int_\Sigma\tH\,d\mu=0$.
For $x,x'$ in the same element of $\cC$, 
we have $\tH(x)-\tH(x')=V(x)-V(x')$. Hence it follows from
    Corollary~\ref{cor:MV} and Remark~\ref{rmk:MV} that
$b_n^{-1}\big|\max_{k\le n}|\tH_k|\big|
    \to_\mu0$.

By the pointwise ergodic theorem $n^{-1/q}\max_{k\le n}g\circ f^k\to0$ a.e.\ for all $L^q$ functions $g:\Sigma\to\R$, $1\le q<\infty$.
Taking $q=1$, it follows that $n^{-1}\max_{k\le n} R\circ f^k\to_\mu0$.
Taking $q=p$,
we have $b_n^{-1}\max_{k\le n} R'\circ f^k
 \ll n^{-1/p}\max_{k\le n} R'\circ f^k\to_\mu0$.

    Finally, $b_n^{-1}R_\Lambda\to_\mu0$ a.e.\ on $(\Lambda,\mu_\Lambda)$.
Since $\mu_\Lambda$ is $T$-invariant,
$b_n^{-1} R_\Lambda \circ T^n \to_{\mu_\Lambda} 0$.
    But $\mu$ is absolutely continuous with respect to $\mu_\Lambda$, so $b_n^{-1} R_\Lambda \circ T^n \to_\mu 0$ too.
\end{proof}

The key ingredient of Step 4 is:
        \begin{prop} \label{prop:SDC}
        $\alpha_\infty(\hW_n\circ T,\hW_n)\le 2|v|_\infty\,b_n^{-1}$.
        \end{prop}

\begin{proof}
Since $\hW_n$ (and hence $\hW_n\circ T$) is a trivial lift,
it follows from Remark~\ref{rmk:lift} that
\[
\alpha_\infty(\hW_n\circ T,\hW_n)
\le |W_n\circ T-W_n|_\infty
\le 2|v|_\infty b_n^{-1},
\]
as required.
\end{proof}

In particular,
        $\alpha_\infty(\hW_n\circ T,\hW_n)\to_\mu 0$.
 Hence we have verified~\cite[Condition~(1)]{Zweimuller07} in Proposition~\ref{prop:SDC}.
Since $\mu$ is absolutely continuous with respect to $\mu_\Lambda$,        
it follows from~\cite[Theorem~1]{Zweimuller07} that
we can replace $\mu$ in Step~3 by $\mu_\Lambda$, thereby
completing the proof of Theorem~\ref{thm:Wconv}.

\section{Tightness of \texorpdfstring{$W_n$ in $p$}{Wn in p}-variation}
\label{sec:tight}

In this section, we prove Theorem~\ref{thm:tight}.
Recall that we have mixing ergodic dynamical systems $(T,\Lambda,\mu_\Lambda)$, $(f,\Sigma,\mu)$ and $(F,Y,\mu_Y)$ where
 $f=T^R$ and $F=f^\tau$,
Here, $R:\Sigma\to\Z^+$ is a first return time and
$\tau:Y\to\Z^+$ is a return time with exponential tails.
Moreover $(F,Y,\mu_Y)$ quotients to a
Gibbs-Markov map $(\bF,\bY,\mu_\bY)$.

Define the induced return time function
\[
R^Y:Y\to\Z^+, \qquad R^Y=\sum_{j=0}^{\tau-1}R\circ f^j.
\]

\begin{prop} \label{prop:phiY}
$\mu_Y(R^Y>t) \sim c\ell(t)t^{-\alpha}$ as $t\to\infty$ 
for some $c>0$.
\end{prop}

\begin{proof}
In this proof, we denote Birkhoff sums under $f$ by $R_n=\sum_{j=0}^{n-1}R\circ f^j$ and so on.
Let $M$ denote the $1\times\MM$ matrix all of whose entries are $1$.
Then $MG_\alpha$ is a nondegenerate totally skewed $1$-dimensional $\alpha$-stable law.
Let $\tR=R-\int_\Sigma R\,d\mu$. 
We claim that 
$b_n^{-1}\tR_n  \to_\mu MG_\alpha$.

The result follows from the claim by
an argument in~\cite[Lemma~5.1(c)]{MV20}.
Indeed, define the piecewise constant function $\tR^Y=\sum_{j=0}^{\tau-1}\tR\circ f^j$.
By Lemma~\ref{lem:induce},
\[
b_n^{-1}\sum_{j=0}^{n-1}\tR^Y \circ F^j\to_{\mu_Y} \bar\tau^{1/\alpha}MG_\alpha.
\]
  By Proposition~\ref{prop:G}, $\mu_Y(\tR^Y>t)\sim c \ell(t) t^{-\alpha}$.
Finally, $R^Y=\tR^Y + \tau\int_\Sigma R\,d\mu$ and $\tau$ has exponential tails, so
  $\mu_Y(R^Y>t)\sim c \ell(t) t^{-\alpha}$.

It remains to prove the claim.
Write $R=Q+Q'$ where $Q=R\bone_{\Sigma\setminus\Sigma_0}$
and $Q'=R\bone_{\Sigma_0}$.
Define $\tQ=Q-\int_\Sigma Q\,d\mu$, $\tQ'=Q'-\int_\Sigma Q'\,d\mu$.

Since $M\omega_i=1$ for all $i$, it follows that
$\tQ=M\tZ$. Evaluating at $t=1$ in Theorem~\ref{thm:Z},
$b_n^{-1}\tZ_n\to_\mu \tL_\alpha(1)=G_\alpha$.
By the continuous mapping theorem,
$b_n^{-1}\tQ_n=b_n^{-1}M\tZ \to_\mu MG_\alpha$.

By Proposition~\ref{prop:R'}, $R\bone_{\Sigma_0}\in L^q$ for some
$q>\alpha$ and hence $\tQ'\in L^q$. Also, $\int_\Sigma\tQ'\,d\mu=0$
and $\tQ'$ is constant on elements of $\cC$.
By Corollary~\ref{cor:MV}, $b_n^{-1}\tQ'_n\to_\mu0$.

The claim for $\tR=\tQ+\tQ'$ now follows by combining the results for $\tQ$ and $\tQ'$.
\end{proof}

Recall that $V:\Sigma\to\R^d$ is given by $V=\sum_{j=0}^{R-1}v\circ T^j$.
Define 
\[
V^Y:Y\to\R^d, \qquad
V^Y=\sum_{\ell=0}^{\tau-1}V\circ f^\ell.
\]
Define the corresponding processes on $(Y,\mu_Y)$,
\[
    W_n^Y\in D([0,1],\R^d),
\qquad
W_n^Y(t)
    = b_n^{-1}\sum_{j=0}^{[nt]-1}V^Y\circ F^j
    .
\]

\begin{lemma} \label{lem:VY}
    $\sup_n \int_{Y}|W_n^{Y}|_{p\var} \, d\mu_{Y}<\infty$
    for all $p>\alpha$.
\end{lemma}

\begin{proof}
    Let $H = V-\sum_{i=1}^\MM P_i(1) \tR_i$
    where $R_i=R\bone_{\Sigma_i}$,
    $\tR_i=R_i-\int_{\Sigma} R_i \, d\mu$.
Then $\int_\Sigma H\,d\mu{=0}$.
Also, $H=V-\sum_{i=1}^\MM P_i(1)R_i+C$ where
$C=\sum_{i=1}^\MM P_i(1)\int_\Sigma R_i\,d\mu\in\R^d$ is a constant.
In particular, $|H|\le |v|_\infty R+|C|$ on $\Sigma_0$ and
$H=V-P_i(1)R+|C|$ on $\Sigma_i$, $i=1,\dots,\MM$. Hence
    by assumptions~\eqref{eq:rv0} and~\eqref{eq:P}, $H \in L^{q}(\Sigma,\mu)$ for some $q>\alpha$. Let $p\in(\alpha,q)$.
    It suffices to show that $\sup_n\int_Y|W_n^Y|_{p\var} \, d\mu_Y <\infty$.

    Write $V^Y=\sum_{i=1}^\MM P_i(1) \tR_i^Y + H^Y:Y\to\R^d$, where
    \[
        \tR_i^Y
        = \sum_{j=0}^{\tau-1}\tR_i \circ f^j
        , \qquad
        H^Y=\sum_{j=0}^{\tau-1} H \circ f^j.
    \]
    Correspondingly, write $W_n^Y = \sum_{i=1}^\MM P_i(1) A_{i,n} + B_n$ on $Y$ where
    \[
        A_{i,n}(t)
        = b_n^{-1} \sum_{j=0}^{[nt]-1} \tR_i^Y\circ F^j
        , \qquad
        B_n(t)
        = b_n^{-1} \sum_{j=0}^{[nt]-1} H^Y\circ F^j
        .
    \]

    Now $\tR_i^Y$ is piecewise constant on $Y$  and hence well-defined
    as a piecewise constant function $\bY$.
    Let
    \[
        \bA_{i,n}(t)
        = b_n^{-1}\sum_{j=0}^{[nt]-1}\tR_i^\bY\circ \bF^j:\bY\to\R
        .
    \]
Let $\sigma=R^\bY+(\int_\Sigma R\,d\mu)\tau$.
    We are going to apply~\cite[Theorem~4.4]{CFKM20}
    (with $\bY$, $\tR_i^\bY$ and $\sigma$ playing the roles of $Z$, $V$ and~$\tau$ in~\cite{CFKM20}).
    By Proposition~\ref{prop:phiY}, $R^\bY$ is regularly varying with exponent $\alpha$. Since $\tau$ has exponential tails,
$\sigma$ is regularly varying with exponent $\alpha$.
    For all $y,y'\in \bY_j$ and all partition elements $\bY_j$,
    \[
        |\tR_i^\bY(y)|\le \sigma(y), \qquad \tR_i^\bY(y)-\tR_i^\bY(y')=0,
    \]
    verifying the condition~\cite[eq.~(4.2)]{CFKM20}.
    Hence it follows from~\cite[Theorem~4.4]{CFKM20} that
    \begin{equation}
        \label{eq:VY:A}
        \sup_n \int_Y |A_{i,n}|_{p\var} \, d\mu_Y
        =\sup_n \int_\bY |\bA_{i,n}|_{p\var} \, d\mu_\bY
        < \infty
        \quad \text{for each } i =1,\dots,\MM
        .
    \end{equation}

    Next,
    by Lemma~\ref{lem:MV},
    \[
        H^Y = m + \chi \circ F - \chi
        ,
    \]
    where $m,\,\chi\in L^p(\mu_Y)$ and $\{m\circ F^{n-j} : 0 \le j \le n\}$ is a martingale difference sequence
    for each $n \ge 1$.
    By a simpler argument than the one in the proof of~\cite[Theorem~4.4]{CFKM20},
    we show that $\lim_{n\to\infty}\int_Y|B_n|_{p\var}^p \, d\mu_Y=0$.
    Indeed, we can write $B_n=M_n+D_n$ where
    \[
        M_n(t) = b_n^{-1}\sum_{j=0}^{[nt]-1} m\circ F^j
        , \qquad
        D_n(t) = b_n^{-1}(\chi\circ F^{[nt]} - \chi)
        .
    \]
    Define also $M_n^{-}(t) = b_n^{-1} \sum_{j=1}^{[nt]} m \circ F^{n-j}$, so that $M_n^{-}(t)$ is a martingale.
    By~\cite[Theorem~2.1]{PisierXu88}, there is $C_p > 0$ depending only on $p$ so that
    \[
        \int |M_n^-|_{p\var}^p \, d\mu_Y
        \le C_p b_n^{-p} \sum_{j=1}^n \int |m\circ F^{n-j}|^p \, d\mu_Y
        = C_p n b_n^{-p} \int |m|^p \, d\mu_Y
        .
    \]
    Hence
    \begin{equation}
        \label{eq:VY:M}
        \int |M_n|_{p\var}^p \, d\mu_Y
        = \int |M_n^-|_{p\var}^p \, d\mu_Y
        \le C_p n b_n^{-p} \int |m|^p \, d\mu_Y
        .
    \end{equation}

    Further, for $1\le j_1< j_2<\cdots<j_k\le n-1$,
    \begin{align*}
        \sum_{i=1}^k|\chi\circ F^{j_i} - & \chi\circ F^{j_{i-1}}|^p
        \le \sum_{i=1}^k \bigl( |\chi|\circ F^{j_i}+|\chi|\circ F^{j_{i-1}} \bigr)^p
        \\
        & \le 2^{p-1} \sum_{i=1}^k \bigl( |\chi|^p\circ F^{j_i}+|\chi|^p\circ F^{j_{i-1}} \bigr)
        \le 2^p \sum_{j=0}^{n-1}|\chi|^p\circ F^j
        .
    \end{align*}
    We deduce that $|D_n|_{p\var}^p \le 2^p b_n^{-p} \sum_{j=0}^{n-1} |\chi|^p \circ F^j$,
    and hence
    \begin{equation}
        \label{eq:VY:D}
        \int_Y |D_n|_{p\var}^p \, d\mu_Y
        \le 2^p n b_n^{-p} \int |\chi|^p \, d\mu_Y
        .
    \end{equation}

    Combining the estimates~\eqref{eq:VY:M} and~\eqref{eq:VY:D} and using
    $\lim_{n \to \infty} n b_n^{-p} = 0$, we see that
    $\lim_{n\to\infty} \int_Y |B_n|_{p\var}^p \, d\mu_Y = 0$.
    This, together with~\eqref{eq:VY:A}, shows that
    $\sup_n \int_{Y}|W_n^{Y}|_{p\var} \, d\mu_Y < \infty$, as required.
\end{proof}

\begin{proof}[Proof of Theorem~\ref{thm:tight}]
Note that
$V^Y=\sum_{j=0}^{R^Y-1}v\circ T^j$.
By Proposition~\ref{prop:phiY}, $R^Y$ is regularly varying.
Applying~\cite[Theorem~5.2]{CFKM20} (with $\Lambda$, $Y$, $W_n^Y$ and $R^Y$ playing the roles of $Y$, $Z$, $\tW_n$ and $\tau$ in~\cite{CFKM20})
and Lemma~\ref{lem:VY},
we obtain that
$|W_n|_{p\var}$ is tight on $(\Lambda,\mu_\Lambda)$.
\end{proof}

\section{Path space -- Proofs}
\label{sec:paths-pf}

In this section we define the space $\DD$ in more detail,
collect its important properties, construct solutions of differential equations driven by elements of $\DD$,
and prove that the solution map is continuous (Theorem~\ref{thm:S}).

We follow the notation and terminology introduced in Section \ref{sec:paths}
and often drop reference to the target space $\R^d$ and write, e.g.\ $D[a,b]$ for $D([a,b],\R^d)$.
We also drop reference to the interval $[a,b]$ when it is clear from the context.
For intervals $I,J\subset \R$, let $\cR_{I,J}$ denote the set of all
continuous increasing bijections $\rho: I\to J$.
We use the shorthand $\cR_{I} = \cR_{I,I}$.

\subsection{The decorated path space \texorpdfstring{$\DD$}{D}}
\label{subsec:pathFuncs}

\begin{defn}\label{def:repara}
Define the space $\cD = \bigcup_{a<b}D[a,b]$ of c{\`a}dl{\`a}g paths parametrised by compact intervals.
For $h_1 \in D[a_1,b_1]$, $h_2 \in D[a_2,b_2]$,
define the \emph{Fr{\'e}chet distance} by
\begin{equation*}
d_F(h_1,h_2) = \inf_{\rho\in \cR_{[a_1,b_1],[a_2,b_2]}} |h_1\circ \rho -h_2|_\infty.
\end{equation*}
We say that $h_1$ is a \emph{reparametrisation} of $h_2$, and write $h_1 \dsim h_2$, if $d_F(h_1,h_2)=0$.
Denote $[h] = \{h'\in \cD : h \dsim h'\}$.
\end{defn}

\begin{rmk}\label{rmk:DF_metric}
$d_F$ defines a metric on $\cD/{\dsim}$.
Indeed, this follows from the fact that
$h_1,h_2$ are reparametrisations
if and only if there exist non-decreasing surjections $\rho_1: [0,1]\to[a_1,b_1]$ and $\rho_2 :[0,1]\to[a_2,b_2]$
such that $h_1\circ \rho_1 = h_2\circ \rho_2$.
See for example~\cite[Exercise~2.5.3]{BBI01}
and~\cite[Proposition~5.3]{BG15}.
\end{rmk}

\begin{lemma}\label{lem:dF_complete}
$\cD/{\dsim}$
is complete under the metric $d_F$.
\end{lemma}

\begin{proof}
We can extract from a Cauchy sequence $[X_n]$ a Cauchy sequence $Y_n\in[X_n]$
in the space $D[0,1]$ with the uniform norm, and apply the fact that this space
is complete.
\end{proof}

For the rest of this subsection	, consider $\delta>0$, a function $\phi:[a,b]\to D([0,1],\R^d)$, and a countable set $\Pi=\{t_1,t_2,\ldots\} \subset[a,b]$ such that $\Pi$ contains the non-stationary points of $\phi$.
We construct functions $\phi^\delta: [a,b+\delta] \to \R^d$ and
$\psi^\delta : [a,b+\delta] \to \R^d\times [a,b]$ as follows.

First we define $\tau_\delta : [a,b] \to [a, b+\delta]$.
Let $0 \le \kappa \le \infty$ be the cardinality of $\Pi$.
If $\kappa=0$, we set $\tau_\delta(t)=t$. Otherwise,
let $r = \sum_{j=1}^\kappa 2^{-j}>0$ and
\[
    \tau_\delta(t)
    = t + \sum_{j=1}^\kappa \frac{\delta 2^{-j}}{r} \bone_{t_j \le t}
    .
\]
Note that $\tau_\delta$ is a strictly increasing c\`adl\`ag function with
$\tau_\delta(t^-) < \tau_\delta(t)$ if and only if $t = t_j$ for some $1 \le j < \kappa+1$.
Moreover, the interval $[\tau_\delta(t_j^-), \tau_\delta(t_j))$ is of length $\delta 2^{-j}/r$.

Let $\phi^\delta(t)=\phi(\min\{t,b\})$ if $\kappa=0$ and otherwise
\begin{equation*}
\phi^\delta(u)
=
\begin{cases} \phi(t)(1) &\mbox{if $u = \tau_\delta(t)$ for some $t \in [a,b]$}, \\ \phi(t_j)\Big(
\frac{u-\tau_\delta(t_j^-)}{\delta 2^{-j}/r}
\Big)
&\mbox{if $u \in [\tau_\delta(t_j^-), \tau_\delta(t_j))$ for some $1 \le j < \kappa+1$}.
\end{cases}
\end{equation*}
We call  $\phi^\delta$ the \emph{$\delta$-extension} of $(\phi,\Pi)$.

Furthermore, let $\tau_\delta^{-1}: [a,b+\delta]\to[a,b]$ be a left inverse of $\tau_\delta$ defined by
\[
    \tau_\delta^{-1}(t) = \inf \{s \in [a,b] : \tau_\delta(s) \geq t\}
    .
\]
This way, $\tau_\delta^{-1}$ is a non-decreasing surjection which
maps each interval $[\tau_\delta(t_j^-),\tau_\delta(t_j)]$ to $t_j$,
and $\tau^{-1}_\delta\circ \tau_\delta = \id$.
Finally, define
\begin{equation}\label{eq:psi_delta}
\psi^\delta(u) = (\phi^\delta(u),\tau_\delta^{-1}(u)).
\end{equation}
We call $\psi^\delta$ the \emph{$\delta$-parametric representation} of $(\phi,\Pi)$.

Recall the spaces $\bar\DD[a,b]$ and $\DD[a,b]$ as in
Definitions~\ref{def:DDbar} and~\ref{def:DD}.
We write $\phi_1\sim\phi_2$ if $\phi_1$ and $\phi_2$ are equivalent in $\bar\DD[a,b]$,
so that $\DD[a,b]=\bar\DD[a,b]/{\sim}$.

\begin{lemma}\label{lem:cadlag_path}
The following statements are equivalent.
\begin{enumerate}[label=(\roman*)]
\item\label{pt:one_cont} At least one of $\psi^\delta$ or $\phi^\delta$ is c{\`a}dl{\`a}g.

\item\label{pt:all_cont} Both $\psi^\delta$ and $\phi^\delta$ are c{\`a}dl{\`a}g.

\item\label{pt:in_DD} $(\phi,\Pi)\in\bar\DD$.
\end{enumerate}
\end{lemma}

\begin{proof}
\ref{pt:all_cont}~$\Rightarrow$~\ref{pt:one_cont} is trivial
while \ref{pt:one_cont}~$\Rightarrow$~\ref{pt:all_cont} follows
from the fact that $\tau_\delta^{-1}$ is continuous.
To prove \ref{pt:in_DD}~$\Rightarrow$~\ref{pt:one_cont}, suppose that $(\phi,\Pi)\in\bar\DD$.
We will prove that $\phi^\delta$ is c{\`a}dl{\`a}g.
If $t\notin \Pi$,
then Lemma~\ref{lem:non-osc} below
implies that $\phi^\delta$ is continuous at $\tau_\delta(t)$.
If $t\in \Pi$, then Lemma~\ref{lem:non-osc} implies that $\phi^\delta$ has a left limit at $\tau_\delta(t^-)$
while the fact that $\phi(t)\in D[0,1]$ implies that $\phi^\delta$ is c{\`a}dl{\`a}g on $[\tau_\delta(t^-),\tau_\delta(t)]$.
Hence~\ref{pt:in_DD}~$\Rightarrow$~\ref{pt:one_cont}.

Conversely, suppose that $\phi\notin\bar\DD$,
in which case either condition~\ref{pt:non_osc} in Definition~\ref{def:DDbar} does not hold or $h: t\mapsto \phi(t)(1)$ is not c{\`a}dl{\`a}g.
If $h$ is not c{\`a}dl{\`a}g, then clearly $\phi^\delta$ is also not c{\`a}dl{\`a}g.
If Definition~\ref{def:DDbar}\ref{pt:non_osc} does not hold,
then there exists $\eps>0$ and a subsequence $n(k)$ such that $\sup_{t\in[0,1]}|\phi(t_{n(k)})(t)-\phi(t_{n(k)})(0)|>\eps$ for all $k\ge 1$.
Furthermore, there exists $t\in[a,b]$ such that either (a) $t_{n(k)} \uparrow t$ as $k\to\infty$ or (b) $t_{n(k)} \downarrow t$ as $k\to\infty$.
In case (a), $\phi^\delta$ does not have a left limit at $\tau_\delta(t^-)$, and in case (b), $\phi^\delta$ does not have a right limit at $\tau_\delta(t)$.
Therefore~\ref{pt:all_cont}~$\Rightarrow$~\ref{pt:in_DD}.
\end{proof}

\begin{lemma}\label{lem:non-osc}
For all $\phi\in\bar\DD[a,b]$, $t\in [a,b]$ and $\eps>0$, there exists $\gamma>0$ such that
\[
\sup_{u\in (t-\gamma,t)} \sup_{s\in[0,1]} |\phi(u)(s)-\phi(t^-)(1)| < \eps,\quad  \sup_{u\in (t,t+\gamma)} \sup_{s\in[0,1]} |\phi(u)(s)-\phi(t)(1)| < \eps.
\]
\end{lemma}

\begin{proof}
This follows from conditions~\ref{pt:cadlag} and~\ref{pt:non_osc} in Definition~\ref{def:DDbar}.
\end{proof}

\begin{rmk}\label{rmk:psi_phi_sim}
For $\phi\in\bar\DD[a,b]$ and $\delta_1,\delta_2>0$, we have
$\psi^{\delta_1} \dsim \psi^{\delta_2}$.
Furthermore,
$\phi$ and $\psi^\delta$ are related via $\psi^\delta = (\phi^\delta, \tau_\delta^{-1})$ and
$\phi(t) \dsim \phi^\delta \mid_{[\tau_\delta(t^-),\tau_\delta(t)]}$
for each $t\in [a,b]$.
Note that the interval $[\tau_\delta(t^-),\tau_\delta(t)]$
is precisely the set of all $y\in[a,b+\delta]$ such that $\tau_\delta^{-1}(y)=t$.
In particular, if $\phi_1,\phi_2 \in\bar\DD[a,b]$, then $\psi_1^\delta \dsim \psi_2^\delta$ for some $\delta>0$ if and only if $\phi_1\sim\phi_2$.
\end{rmk}

\begin{rmk}\label{rmk:error_1}
Lemma~\ref{lem:cadlag_path} remains true (with essentially the same proof)
when one replaces ``c{\`a}dl{\`a}g'' by ``continuous'' in \ref{pt:one_cont}--\ref{pt:all_cont} and $\bar\DD$ by $\bar\DD_c$ defined as the space of all pairs $(\phi,\Pi)\in\DD$ such that $\phi : [a,b] \to C[0,1]$ and $\phi(t)(0) = \phi(t^-)(1)$ in \ref{pt:in_DD}.

We point out a minor error in~\cite{Chevyrev18, CF19, CFKM20}:
the constructions therein implicitly claimed that $\phi^\delta$ is in $C[a,b+\delta]$ whenever $\phi \in \bar\DD_c$ without assuming condition~\ref{pt:non_osc} in Definition~\ref{def:DDbar}, which is false.
However, the results in those articles always deal with decorated paths (called path functions therein)
having further nice properties, such as \emph{endpoint continuity} or \emph{finiteness of $p$-variation}, for which Definition~\ref{def:DDbar}\ref{pt:non_osc} holds and thus
$\phi^\delta\in C[a,b+\delta]$.
\end{rmk}

\begin{rmk}[Intrinsic definition of $\phi^\delta$] \label{rmk:xPhiwelldefined} 
The construction of $\phi^\delta$ involves several arbitrary choices: the set
$\Pi = \{t_1, t_2, \ldots\}$ with the enumeration of its points and the sequence $2^{-j}$.
If $\bar \phi^\delta$ is constructed similarly, but using another $\Pi$
and another summable sequence, then $\phi^\delta$ and $\bar \phi^\delta$ are reparametrisations of each other.
See Lemma~\ref{lem:limExists} where these arbitrary choices are in essence removed.
The same remarks apply to $\psi^\delta$.
\end{rmk}

\subsection{A generalisation of Skorokhod's \texorpdfstring{$\cM_1$}{M1} topology}
\label{subsec:SkorM1}

We fix $p\in [1,\infty]$.
Recall the semi-norm $|\cdot|_{p\var}$ and space $D^{p\var}[a,b]$ from Section~\ref{sec:pvar}.
For later use, we record the interpolation estimate
\begin{equation}
    \label{eq:interpolation}
    |f|_{q\var}\le |f|_{{\infty\var}}^{1-p/q}|f|_{p\var}^{p/q}
    \quad \text{for} \quad
    1 \le p \le q \le \infty
    .
\end{equation}
For $f: I\to \R$ denote
\[
    \disc{f} = \sup_{x\in I}|f(x)-x|.
\]
Remark that if $g: J\to I$, then $\disc{f\circ g} \le \disc{f}+\disc{g}$.

\begin{lemma}\label{lem:diff_repar}
Consider $\delta>0$, $g_i\in D[a,b]$ and $h_i \in D[a,b+\delta]$ for $i=1,2$
such that $h_i=g_i\circ\rho_i$ for some $\rho_i\in\cR_{[a,b+\delta],[a,b]}$.
Then
\begin{equation}\label{eq:simga_inf_lambda}
|\sigma_{\infty}(g_1,g_2) - \sigma_\infty(h_1,h_2)| \le \disc{\rho_1} + \disc{\rho_2}.
\end{equation}
If $g_i,h_i \in D^{p\var}[a,b]$, then~\eqref{eq:simga_inf_lambda}
holds with $\sigma_\infty$ replaced by $\sigma_{p\var}$.
\end{lemma}

\begin{proof}
By definition of $\sigma_\infty$,
\begin{align*}
|\sigma_{\infty}(g_1,g_2) - \sigma_\infty(h_1,h_2)|
&=
\big|
\inf_{\omega_1 \in \cR_{[a,b]}} \max\{\disc{\omega_1} , |g_1\circ \omega_1-g_2|_\infty\}
\\
&\qquad
- \inf_{\omega_2 \in \cR_{[a,b+\delta]}} \max\{\disc{\omega_2} , |g_1\circ\rho_1\circ\omega_2\circ\rho_2^{-1}-g_2|_\infty\}
\big|
\end{align*}
from which \eqref{eq:simga_inf_lambda} follows
upon making the substitution $\omega_1 =\rho_1\circ\omega_2\circ\rho_2^{-1}$
and using
$\disc{\rho_1\circ\omega_2\circ\rho_2^{-1}} \le \disc{\rho_1}+\disc{\omega_2}+\disc{\rho_2}$.
An identical argument shows the final claim.
\end{proof}

\begin{lemma}\label{lem:diff_deltas}
For all $\delta_2>\delta_1>0$
and $(\phi ,\Pi)\in\bar\DD[a,b]$,
there exists $\rho\in \cR_{[a,b+\delta_1],[a,b+\delta_2]}$
such that $\phi^{\delta_2}\circ\rho = \phi^{\delta_1}$
and $\disc{\rho} = \delta_2-\delta_1$.
\end{lemma}

\begin{proof}
    Obvious from the construction of $\phi^\delta$.
\end{proof}

\begin{lemma}\label{lem:p-var_indep}
Consider $\delta_1,\delta_2>0$ and
$\phi_1,\phi_2 \in \bar\DD[a,b]$ with $\phi_1\sim\phi_2$.
Then $|\phi_1^{\delta_1}|_{p\var;[a,b+\delta_1]} = |\phi_2^{\delta_2}|_{p\var;[a,b+\delta_2]}$.
\end{lemma}

\begin{proof}
This is immediate from the fact that $\phi_1^{\delta_1}$ is a reparametrisation of $\phi_2^{\delta_2}$.
\end{proof}

Recall the space  $\DD^{p\var}[a,b]$ from~\eqref{eq:DD_pvar}.

\begin{lemma}\phantomsection \label{lem:limExists}
\begin{enumerate}[label=(\roman*)]
\item\label{pt:alpha_infinity} The
limit $
\alpha_{\infty}(\phi_1,\phi_2) = \lim_{\delta \to 0} \sigma_{\infty} (\phi_1^\delta,\phi_2^\delta)
$
exists and is well-defined for all $(\phi_1,\Pi_1), (\phi_2,\Pi_2) \in \DD[a,b]$,
i.e.\ is independent of the choice of the series $\sum_{j=1}^\infty 2^{-j}$
used in Section~\ref{subsec:pathFuncs}, the choice of $\Pi_1$ and $\Pi_2$,
the respective enumeration of points, or the
parametrisations of $\phi_1(t)$, $\phi_2(t)$ for each $t$.

\item\label{pt:alpha_p-var} The limit
$
\alpha_{p\var}(\phi_1,\phi_2) = \lim_{\delta \to 0} \sigma_{p\var} (\phi_1^\delta,\phi_2^\delta)
$
exists and is well-defined for $\phi_1,\phi_2\in \DD^{p\var}[a,b]$ in the same sense as in \ref{pt:alpha_infinity}.
\end{enumerate} 
\end{lemma}

\begin{proof}
\ref{pt:alpha_infinity}
Lemmas~\ref{lem:diff_repar}
and~\ref{lem:diff_deltas} imply that, for $\delta_2>\delta_1>0$ and $\phi_1,\phi_2\in\bar\DD[a,b]$,
\[
|\sigma_{\infty}(\phi_1^{\delta_2},\phi_2^{\delta_2}) - \sigma_{\infty}(\phi_1^{\delta_1}, \phi_2^{\delta_1})|
\le
2(\delta_2-\delta_1),
\]
from which the existence of the limit follows.
The fact that the limit is independent of the series $\sum_j 2^{-j}$, the enumeration of non-stationary points,
and representatives of equivalence classes
is clear.
\ref{pt:alpha_p-var} follows in an identical manner.
\end{proof}

\begin{lemma}\label{lem:alpha_inf_dF}
Let $\delta>0$ and let $\psi_1^\delta,\psi_2^\delta:[a,b+\delta]\to \R^d\times[a,b]$ be the $\delta$-parametric representations of $\phi_1,\phi_2\in \bar\DD[a,b]$ as in~\eqref{eq:psi_delta}.
Then
\[
\alpha_\infty(\phi_1,\phi_2)
=
d_F(\psi_1^\delta,\psi_2^\delta),
\]
where we equip $\R^d\times\R$ with the norm $\max\{|h|,|t|\}$ for $(h,t)\in\R^d\times\R$.
\end{lemma}

\begin{proof}
Observe that $d_F(\psi_1^\delta,\psi_2^\delta)$ does not depend on $\delta$, hence, where convenient,
we can suppose that $\delta$ is arbitrarily small.

Suppose $\rho\in \cR_{[a,b+\delta]}$.
Let $\tau_{1,\delta},\tau_{2,\delta}$ and $\tau_{1,\delta}^{-1},\tau_{2,\delta}^{-1}$
be the time-changes associated to $\phi_1,\phi_2$ respectively as in Section~\ref{subsec:pathFuncs}.
Then
\[
\disc{\tau_{1,\delta}^{-1}\circ\rho\circ\tau_{2,\delta}}
\le
\sup_{t\in[a,b+\delta]}|\tau_{1,\delta}^{-1}(\rho(t))-\tau^{-1}_{2,\delta}(t)|
\le \disc{\tau^{-1}_{1,\delta}}+\disc{\tau^{-1}_{2,\delta}}
+\disc{\rho}
,
\]
where in the first bound we used that $\tau^{-1}_{2,\delta}$ is surjective.
Therefore
\begin{align*}
|\psi_1^\delta\circ\rho-\psi_2^\delta|_\infty
&=
\sup_{t\in[a,b+\delta]}
\max
\{|\phi_1^\delta(\rho(t))- \phi_2^\delta(t)| , |\tau_{1,\delta}^{-1}(\rho(t))-\tau^{-1}_{2,\delta}(t)|\}
\\
&\ge \max \bigl\{ |\phi_1^\delta\circ\rho-\phi_2^\delta|_\infty,
    \disc{\tau_{1,\delta}^{-1}\circ\rho\circ\tau_{2,\delta}}
\bigr\}.
\end{align*}
Furthermore, $\disc{\tau_{1,\delta}^{-1}\circ\rho\circ\tau_{2,\delta}}\to\disc{\rho}$ as $\delta\to0$,
and thus $d_F(\psi_1^\delta,\psi_2^\delta)\ge \alpha_\infty(\phi_1,\phi_2)$.
Conversely,
\[
|\psi_1^\delta\circ\rho-\psi_2^\delta|_\infty
\le
\max
\{|\phi_1^\delta\circ\rho-\phi_2^\delta|_\infty,
\disc{\rho} + \disc{\tau^{-1}_{1,\delta}} + \disc{\tau^{-1}_{2,\delta}}\}
\]
which proves $d_F(\psi_1^\delta,\psi_2^\delta)\le \alpha_\infty(\phi_1,\phi_2)$
since $\disc{\tau^{-1}_{1,\delta}} + \disc{\tau^{-1}_{2,\delta}} \to 0$ as $\delta\to0$.
\end{proof}

\begin{thm}\label{thm:alpha}
$\DD[a,b]$ equipped with $\alpha_\infty$ is a complete separable metric space.
\end{thm}

\begin{proof}
We first prove that $\alpha_\infty$ is a metric.
It is clear that $\alpha_\infty$ satisfies the triangle inequality,
so following Lemma~\ref{lem:limExists}
it remains only to show that if $\alpha_\infty(\phi_1,\phi_2)=0$ for some $\phi_1,\phi_2\in\bar\DD[a,b]$,
then $\phi_1\sim\phi_2$.
This in turn follows from Lemma~\ref{lem:alpha_inf_dF} and Remarks~\ref{rmk:psi_phi_sim} and~\ref{rmk:DF_metric}.

Next, if $\phi_n \in\DD$ is a Cauchy sequence for $\alpha_\infty$,
then $\psi^\delta_n \in \cD$ is a Cauchy sequence for $d_F$ by Lemma~\ref{lem:alpha_inf_dF}.
By Lemma~\ref{lem:dF_complete}, $\cD / {\dsim}$ is complete under $d_F$,
therefore $\lim_{n\to\infty}d_F(\psi^\delta_n,\psi)\to 0$
for some $\psi\in\cD$.
Since the second component of each $\psi^\delta_n$ is a non-decreasing surjection,
the same holds for $\psi$.
We can then readily define $\phi\in\DD$ with $\delta$-parametric representation $\psi^\delta$ such that $\psi\dsim \psi^\delta$,
and thus $\phi_n\to\phi$ in $\alpha_\infty$ due to Lemma~\ref{lem:alpha_inf_dF}.
Hence $(\DD,\alpha_\infty)$ is complete.

Finally, note that $(\cD,d_F)$ is a separable metric space (this follows, e.g.\ from separability of the Skorokhod space).
Again by Lemma~\ref{lem:alpha_inf_dF},
we can identify $(\DD,\alpha_\infty)$ with a subspace of $(\cD,d_F)$,
and thus $(\DD,\alpha_\infty)$ is separable.
\end{proof}

\begin{proof}[Proof of Theorem~\ref{thm:comp}]
By Theorem~\ref{thm:alpha}, $(\DD[a,b],\alpha_\infty)$ is complete.
Furthermore, recall the embedding $\iota : D[a,b]\to \DD[a,b]$ from Definition~\ref{def:iota}.
It is easy to see that $\alpha_\infty=\sigma_\infty$ on $\iota D[a,b]$, thus $(\iota D[a,b],\alpha_\infty)$ is isometric to $(D[a,b],\sigma_\infty)$ (Remark~\ref{rmk:lift}).
Finally, $\iota D[a,b]$ is dense in $\DD[a,b]$ since $\alpha_\infty(\phi, \iota \phi^\delta\circ\rho)\leq2\delta$ where $\rho$ is the linear reparametrisation $\rho:[a,b+\delta]\to[a,b]$.
\end{proof}

\subsection{Differential equations driven by decorated paths}
\label{sec:RDEs}

Consider $p\in [1,2)$, $\xi\in\R^m$, and $A,B$ as in Section~\ref{sec:DEs}.
In this section we make precise the meaning of and study the differential equation
\begin{equation*}
d X = A(X)\,dt+B(X)\,d \phi, \quad X(a)=\xi,
\end{equation*}
where $\phi\in \DD^{p\var}([a,b],\R^d)$
and which we now pose on a general interval $[a,b]$.

We first recall and make precise the construction of $X$ as an element of $\DD^{p\var}([a,b],\R^m)$.
Recall $(\omega,\Pi) = ((\iota\id, \phi),\Pi)\in \bar\DD^{p\var}([a,b],\R^{1+d})$
from Section~\ref{sec:DEs}
where $\id :[a,b]\to[a,b]$ is the identity map.
By \cite[Theorem~3.6]{CFKMZ22},\footnote{There is an assumption  in \cite[Theorem~3.6]{CFKMZ22}
that the drift has finite $p/2$-variation; we can drop this assumption here as
we are working in the Young regime $p\in [1,2)$, see also \cite[Theorem~3.2 \& Remark~6.1]{FZ18}.}
the Young ODE $d X^\delta = (A,B)(X^\delta) \,d \omega^\delta$
admits a unique solution $X^\delta\in D^{p\var}([a,b],\R^m)$ since $p\in[1,2)$.
We define $X\colon [a,b]\to D^{p\var}([0,1],\R^m)$
by taking $X(t)$ as
the linear reparametrisation of $X^\delta |_{[\tau_\delta(t^-),\tau_\delta(t)]}$,
where $\tau_\delta$ is associated to $(\omega,\Pi)$ as in Section~\ref{subsec:pathFuncs}.

The main continuity result on ODEs driven by decorated paths in the Young regime is the following.

\begin{prop}\label{prop:stability_ODE}
The map $\cS: (\xi,\phi) \mapsto X$ is well-defined as a map $\R^m\times \DD^{p\var}([a,b],\R^d) \to \DD^{p\var}([a,b],\R^m)$
and is locally H{\"o}lder continuous when the spaces $\DD^{p\var}$ on both sides are equipped with $\alpha_{p\var}$.
\end{prop}

\begin{proof}
For $\xi,\phi$ as above, we first verify that $(X,\Pi)\in \DD^{p\var}([a,b],\R^m)$.
If $t\notin\Pi$, then $t$ is a continuity point of $\tau_\delta$, hence $[\tau_\delta(t^-),\tau_\delta(t)]$ is a singleton and $X(t) \equiv X^\delta(\tau_\delta(t))$ is constant.
Therefore $\Pi$ contains all the non-stationary points of $X$.
Furthermore $X(t)(1) = X^\delta\circ \tau_\delta(t)$ which is c{\`a}dl{\`a}g since $X^\delta$ is c{\`a}dl{\`a}g and $\tau_\delta$ is c{\`a}dl{\`a}g and increasing. 
The fact that $\sup_{s\in[0,1]}|X(t)(s)-X(t)(0)|>\eps>0$
for at most finitely many $t\in\Pi$
follows from $|X^\delta|_{p\var}<\infty$.
Therefore $(X,\Pi)\in \DD([a,b],\R^m)$ and thus $(X,\Pi)\in \DD^{p\var}([a,b],\R^m)$.

Next, we recall that ODEs commute with reparametrisations:
if $Y_i\in D^{p\var}([a,b],\R^m)$ for $i=1,2$ solves $dY_i=(A,B)(Y_i)\,du_i$ for $u_i\in D^{p\var}([a,b],\R^{1+d})$ with first component of bounded variation, and $u_1\circ\rho = u_2$ for some $\rho\in\cR_{[a,b]}$,
then $Y_1\circ\rho = Y_2$.
It readily follows that if $X'$ is a solution associated to $(\xi,\phi')$ with $\omega'=(\iota\id,\phi')\sim \omega$, then $X\sim X'$.
Moreover, $(\iota\id,\phi')\sim (\iota\id,\phi)$ if and only if $\phi'\sim\phi$, therefore $\cS$ is well-defined.

It remains to prove that $\cS$ is locally H{\"o}lder continuous.
Consider $\xi_i,\phi_i,\omega_i,X_i$ as above for $i=1,2$
and take $\Pi$ to contain all non-stationary points of $\phi_1$, $\phi_2$.
Remark that $\omega^\delta_i= (g^\delta,\phi_i^\delta) \in D^{p\var}([a,b+\delta],\R^{1+d})$
for a non-decreasing surjection $g^\delta:[a,b+\delta]\to [a,b]$
with $\disc{g^\delta}\leq \delta$.

Recalling that $A\in C^{\beta}(\R^m,\R^m)$ for $\beta>1$,
take $q\in (1,\beta)$.
Since ODEs commute with reparameterisations, 
$X_1^\delta\circ\rho$ solves the ODE
$dX = (A,B)(X)\,d(\omega_1^\delta\circ\rho)$
for any $\rho\in\cR_{[a,b+\delta]}$.
Therefore, by~\cite[Theorem~3.6]{CFKMZ22},
\begin{equation*}
|X_1^\delta\circ\rho - X_2^\delta|_{p\var} \le C (|g^\delta\circ\rho - g^\delta|_{q\var} + |\phi^\delta_1\circ\rho-\phi^\delta_2|_{p\var} + |\xi_1-\xi_2|),
\end{equation*}
where $C$ here and below depends on $\phi_i$, $\xi_i$ only through $|\xi_1-\xi_2|+\sum_{i=1}^2|\phi_i|_{p\var}$.
By~\eqref{eq:interpolation},
\[
|g^\delta\circ\rho - g^\delta|_{q\var}
\leq 
|g^\delta\circ\rho - g^\delta|_{\infty\var}^{1-1/q}
\;
|g^\delta\circ\rho - g^\delta|_{1\var}^{1/q}.
\]
Using $|\cdot|_{\infty\var}\leq 2|\cdot|_\infty$ and $\disc{g^\delta}\leq \delta$, we obtain $|g^\delta\circ\rho - g^\delta|_{\infty\var} \leq 4\delta + 2\disc{\rho}$.
Furthermore, $|g^\delta\circ\rho - g^\delta|_{1\var} \le |g^\delta\circ\rho|_{1\var} + |g^\delta|_{1\var} \le 2(b-a)$,
and thus
\begin{align*}
\alpha_{p\var}(X_1,X_2) &\leq C\lim_{\delta\to0}
\inf_{\rho\in\cR}
\max
\{
\disc{\rho},  \disc{\rho}^{1-1/q}
+
|\phi^\delta_1\circ\rho-\phi^\delta_2|_{p\var}
+|\xi_1-\xi_2|
\}
\\
&\leq C \bigl(
\alpha_{p\var}(\phi_1,\phi_2)+\alpha_{p\var}(\phi_1,\phi_2)^{1-1/q}
+|\xi_1-\xi_2|
\bigr).
\qedhere
\end{align*}
\end{proof}

To get control on $\alpha_{p\var}$, there is a useful interpolation inequality.

\begin{lemma}\label{lem:interpol_sigma}
Let $g,h : [a,b]\to \R^d$ (not necessarily c\`adl\`ag)
and $1\le p\le q\le \infty$.
Then
\[
\sigma_{q\var}(g,h) \le \{1+(|g|_{p\var} + |h|_{p\var})^{p/q} \} \{\sigma_{\infty\var}(g,h)^{1-p/q} + \sigma_{\infty\var}(g,h)\}.
\]
\end{lemma}

\begin{proof}
It follows from~\eqref{eq:interpolation} that
\begin{align*}
\sigma_{q\var}(g,h) &= \inf_\rho \max\{\disc{\rho} , |g\circ \rho - h|_{q\var} + |g(a)-h(a)|\}
\\
&\le \inf_\rho
\max \{\disc{\rho} ,
|g\circ \rho - h|_{p\var}^{p/q} \;
|g\circ \rho -h|_{\infty\var}^{1-p/q} + |g(a)-h(a)|\}
\\
&\le \inf_\rho
\max \{\disc{\rho} , (|g|_{p\var} + |h|_{p\var})^{p/q} \,
|g\circ \rho-h|_{\infty\var}^{1-p/q} + |g(a)-h(a)|\}
\\
&\le \{1+(|g|_{p\var} + |h|_{p\var})^{p/q} \} \{\sigma_{\infty\var}(g,h)^{1-p/q} + \sigma_{\infty\var}(g,h)\}.\qedhere
\end{align*}
\end{proof}

\begin{lemma}\label{lem:interpol_alpha}
For $\phi_1,\phi_2\in \DD^{q\var}$ and $1\le p\le q\le \infty$,
\[
\alpha_{q\var}(\phi_1,\phi_2) \le \{1+(|\phi_1|_{p\var} + |\phi_2|_{p\var})^{p/q} \} \{\alpha_{\infty\var}(\phi_1,\phi_2)^{1-p/q} + \alpha_{\infty\var}(\phi_1,\phi_2)\}.
\]
\end{lemma}

\begin{proof}
We apply Lemma~\ref{lem:interpol_sigma} to $\phi_1^\delta$ and $\phi_2^{\delta}$ for each $\delta>0$
and use the definition of $\alpha_{q\var}$ and $\alpha_{\infty\var}$.
\end{proof}

The following is a minor extension of the continuous mapping theorem.

\begin{lemma}\label{lem:cont_mapping}
Let $E,F$ be metric spaces equipped with their Borel $\sigma$-algebras and let $H: E\to F$ be a mapping.
Suppose that there exist closed sets $D_k\subset E$ for $k\ge 1$ such that $D_k \subset D_{k+1}$ and
$\bigcup_{k\ge 1} D_k = E$, and that $H$ is continuous on every $D_k$.
Suppose that $X_n$ are $E$-valued random variables such that $\lim_{k\to\infty} \sup_n \P[X_n\notin D_k] = 0$.
Finally, suppose that $X_n \to_w X$ in $E$ as $n\to\infty$.
Then $H(X_n)\to_w H(X)$ in $F$ as $n\to\infty$.
\end{lemma}

\begin{proof}
Let $f: F\to \R$ be a lower semicontinuous function such that $|f| \le M$.
We claim that the function $H_k= \bone_{D_k}f\circ H + \bone_{D_k^c} M$ is lower semicontinuous on $E$.
We suppose that $x_n\to x$ in $E$ and we show that $\liminf_{n \to \infty} H_k(x_n) \geq H_k(x)$. 
The verification is straightforward if $\liminf_{n \to \infty} H_k(x_n)=M$.
Suppose now $\liminf_{n \to \infty} H_k(x_n)<M$.
Let
$x_{n_m}$ be a subsequence such that $\liminf_{n \to \infty} H_k(x_n) = \lim_{m \to \infty} H_k(x_{n_m})$.
Note that $x_{n_m}$ is eventually in $D_k$ and thus $x \in D_k$ by the assumption that $D_k$ is closed.
Then by continuity of $H$ on $D_k$ and lower semicontinuity of $f$,
\[
    \liminf_{n \to \infty} H_k(x_n)
    = \lim_{m \to \infty} H_k(x_{n_m})
    = \lim_{m \to \infty} f\circ H(x_{n_m})
    \geq f \circ H(x)
    = H_k(x)
    .
\]
The claim is proved.

Next, by the Portmanteau lemma,
\[
\liminf_{n\to\infty} \E[H_k(X_n)] \ge \E[H_k(X)].
\]
Moreover, $H_k$ is decreasing in $k$ with $H_k \to f\circ H$ pointwise, and therefore 
\[
\lim_{k\to\infty} \E[H_k(X)] = \E[f\circ H(X)].
\]
On the other hand,
\[
    \lim_{k\to\infty}\sup_n \bigl( \E[H_k(X_n)]-\E[f\circ H(X_n)] \bigr) = 0
\]
by the assumption that $\lim_{k\to\infty} \sup_n \P(X_n\notin D_k) = 0$.
Combining these facts, we obtain
\[
    \liminf_{n \to \infty} \E[f\circ H(X_n)] \ge \E[f\circ H(X)].
\]
Since this holds for all bounded lower semicontinuous functions $f$,
it follows that $H(X_n) \to_w H(X)$ in $F$ by the Portmanteau lemma.
\end{proof}

We are now ready to prove Theorem~\ref{thm:S}.

\begin{proof}[Proof of  Theorem~\ref{thm:S}]
The fact that $|\phi|_{p\var}<\infty$ follows from lower semicontinuity of $p$-variation and the Portmanteau lemma
(more specifically, we use that $\bone_{|\phi|_{p\var}>n}$ is lower semicontinuous).
Consider now the metric spaces $E=(\DD^{p\var}([a,b],\R^d),\alpha_\infty)$
and $F=(\DD^{p\var}([a,b],\R^m),\alpha_{q\var})$ with $q\in (p,\gamma)$.
Recall that $B\in C^\gamma(\R^m,\R^{m\times d})$ for $\gamma>p$.
By stability of Young ODEs  (Proposition~\ref{prop:stability_ODE}), interpolation (Lemma~\ref{lem:interpol_alpha}),
and the elementary inequality
$|f(a)| + |f|_{\infty\var} \le 3|f|_{\infty}$
and thus
$\alpha_{\infty\var} \le 3\alpha_{\infty}$,
the solution map $\cS$ maps every subset $D_k = \{\phi\in E\,:\, |\phi|_{p\var} \le k\}$ continuously into $F$.
By lower semicontinuity of $p$-variation, each set $D_k$ is closed in $E$,
and, by the assumption that $|\phi_n|_{p\var}$ is tight,
$\lim_{k\to\infty} \sup_n \P[\phi_n\notin D_k] = 0$.
Finally $\phi_n \to_w \phi$ in $E$ by assumption.
It follows from Lemma~\ref{lem:cont_mapping} that $\cS(\phi_n)\to_w \cS(\phi)$ in $F$ as $n\to\infty$ as claimed.
\end{proof}

\paragraph{Acknowledgements}
We are grateful to the referees for their very careful reading of the paper and their helpful suggestions.
I.C. acknowledges support by the Engineering and Physical Sciences Research Council via the New Investigator Award EP/X015688/1
and would like to thank the University of Warwick for its hospitality during which this work was concluded.
A.K. has been partially supported by EPSRC grant EP/V053493/1, and is grateful to Wael Bahsoun for support.
We also thank ICMS, Edinburgh, for its hospitality during which part of this work was carried out.

\end{document}